\documentclass[twoside,11pt]{article}

\usepackage{jmlr2e}

\usepackage[utf8]{inputenc}
\usepackage[T1]{fontenc}
\usepackage{url}
\usepackage{booktabs}
\usepackage{amsfonts}
\usepackage{dsfont}
\usepackage{amsmath}
\usepackage{amssymb}
\usepackage{mathtools}
\usepackage{multirow}
\usepackage{textcomp,gensymb}
\usepackage{array}
  \newcolumntype{P}[1]{>{\centering\arraybackslash}p{#1}}
\usepackage{hyperref}
  \hypersetup{hidelinks}
\usepackage{microtype}

\usepackage{lastpage}
\jmlrheading{25}{2024}{1-\pageref{LastPage}}{11/23; Revised 7/24}{7/24}{23-1553}{Roberto I. Oliveira, Paulo Orenstein, Thiago Ramos and João Vitor Romano}

\ShortHeadings{Split CP and Non-Exchangeable Data}{Oliveira, Orenstein, Ramos and Romano}
\firstpageno{1}

\DeclareMathOperator{\train}{train}
\let\cal\relax
\DeclareMathOperator{\cal}{cal}
\DeclareMathOperator{\test}{test}
\newcommand{\R}{\mathbb{R}}
\newcommand{\Z}{\mathbb{Z}}
\newcommand{\N}{\mathbb{N}}
\newcommand{\sA}{\mathcal{A}}
\newcommand{\sF}{\mathcal{F}}
\newcommand{\sS}{\mathcal{S}}
\newcommand{\sT}{\mathcal{T}}
\newcommand{\sW}{\mathcal{W}}
\newcommand{\sX}{\mathcal{X}}
\newcommand{\sY}{\mathcal{Y}}
\newcommand{\sZ}{\mathcal{Z}}
\renewcommand{\P}{\mathbb{P}}
\newcommand{\E}{\mathbb{E}}
\newcommand{\V}{\mathrm{Var}}
\newcommand{\VC}{\mathrm{VC}}
\DeclareMathOperator{\TV}{TV}
\DeclarePairedDelimiter{\braces}{\lbrace}{\rbrace}
\newcommand{\Ind}{\operatorname{\mathds{1}}\braces*}
\DeclareMathOperator{\Avg}{Avg}

\begin{document}

\title{Split Conformal Prediction and Non-Exchangeable Data}

\author{\name Roberto I. Oliveira \email rimfo@impa.br \\
       \addr IMPA, Rio de Janeiro, Brazil.
       \AND
       \name Paulo Orenstein \email pauloo@impa.br \\
       \addr IMPA, Rio de Janeiro, Brazil.
       \AND
       \name Thiago Ramos \email thiagorr@impa.br \\
       \addr IMPA, Rio de Janeiro, Brazil.
       \AND
       \name João Vitor Romano \email joao.vitor@impa.br \\
       \addr IMPA, Rio de Janeiro, Brazil.
}

\editor{Chris Oates}

\maketitle

\begin{abstract}%
Split conformal prediction (CP) is arguably the most popular CP method for
uncertainty quantification, enjoying both academic interest and widespread
deployment. However, the original theoretical analysis of split CP makes the
crucial assumption of data exchangeability, which hinders many real-world
applications. In this paper, we present a novel theoretical framework based on
concentration inequalities and decoupling properties of the data, proving that
split CP remains valid for many non-exchangeable processes by adding a small
coverage penalty. Through experiments with both real and synthetic data, we
show that our theoretical results translate to good empirical performance under
non-exchangeability, e.g., for time series and spatiotemporal data. Compared to
recent conformal algorithms designed to counter specific exchangeability
violations, we show that split CP is competitive in terms of coverage and
interval size, with the benefit of being extremely simple and orders of
magnitude faster than alternatives.
\end{abstract}


\section{Introduction}

Conformal prediction (CP), introduced by \citet{vovk2005algorithmic}, is a set of techniques for quantifying uncertainty in the predictions of any model, under very general assumptions on the data-generating distribution. CP yields finite-sample coverage guarantees of many kinds, and has generated much recent interest \citep{vovk2008tutorial,lei2018distribution,romano2019conformalized,angelopoulos2023gentle,cauchois2021prediction}.

A concrete and popular formulation of CP is split conformal prediction \citep{papadopoulos2002inductive,lei2018distribution}. Consider a regression setting where the data is a random sample $(X_i,Y_i)_{i=1}^n$ of covariate and response pairs $(X_i,Y_i)\in\sX\times \sY$. Split CP proceeds as follows: (a) partition the data indices in three parts: a training set $I_{\train}$, a calibration set $I_{\cal}$ and a test set $I_{\test}$, each with sizes $n_{\train}$, $n_{\cal}$ and $n_{\test}$; (b) train a nonconformity score $\widehat{s}_{\train}\colon\sX\times \sY \to \R$, for example the residual $\widehat{s}_{\train}(x,y)=|y-\widehat{\mu}(x)|$ of an arbitrary model $\widehat{\mu}$ trained on $(X_i, Y_i)_{i\in I_{\train}}$; (c) compute the empirical $(1-\alpha)$-quantile $\widehat{q}_{1-\alpha}$ of $\{\widehat{s}_{\train}(X_i, Y_i)\}_{i \in I_{\cal}}$; and (d) for each $i\in I_{\test}$, define a confidence set
\begin{equation*}
    C_{1-\alpha}(X_i) \coloneqq \{y\in\sY\,:\,\widehat{s}_{\train}(X_i,y)\leq \widehat{q}_{1-\alpha}\}.
\end{equation*}
We note in passing that split CP itself does not depend on test data; $I_{\test}$ is included for notational ease and theoretical convenience, as the guarantees of the method pertain to unseen data. If the data $(X_i, Y_i)_{i=1}^{n}$ is exchangeable, then the usual theory of conformal prediction guarantees that the sets $C_{1-\alpha}(X_i)$ have good marginal coverage over the test set; that is, for any $i\in I_{\test}$,
\begin{equation}\label{eq:marginaliid}
    \P[Y_i\in C_{1-\alpha}(X_i)]\geq 1-\alpha - \eta,
\end{equation}
where $\eta=(1-\alpha)/(n_{\cal}+1)$. Equivalently, one can take the lower bound to be $1-\alpha$ by employing $C_{1-\alpha+\eta}$ instead. Additionally, for independent and identically distributed (iid) data and $\eta\gg 1/\min\{n_{\cal},n_{\test}\}$, \citet{lei2018distribution} prove empirical coverage over the test set; that is, for some positive constant $c>0$ and $\delta=\exp(-c\,\eta^2\,\min\{n_{\cal},n_{\test}\})$,
\begin{equation}\label{eq:empiricaliid}
\P\left[\frac{1}{n_{{\test}}}\sum_{i\in I_{\test}}\,\Ind{Y_i\in C_{1-\alpha}(X_i)}\geq 1-\alpha-\eta\right]\geq 1-\delta.
\end{equation}

Unfortunately, the results above are strongly reliant on the data exchangeability. Similarly, most guarantees from the classical theory of CP do not apply to several important data processes. For instance, in a time series context it is typical to predict the current value $Y_i$ via a time-lagged vector $X_i=(Y_{i-j})_{j=1}^p$, and the resulting process $(X_i,Y_i)_{i=1}^n$ is typically far from exchangeable. Spatial models and shifting data distributions will also break the exchangeability assumption.

Several recent papers have tried to address these issues \citep{chernozhukov2018blocks, xu2021dynamic, xu2023conformal, jensen2022ensemble, gibbs2021adaptive, gibbs2024conformal, barber2023conformal,zaffran2022adaptive,feldman2022risk}, but they generally require the introduction of new CP algorithms specifically tailored to different types of non-exchangeability. Some of these are very computationally intensive, while others only possess asymptotic guarantees. A third group requires adaptativity---that is, recalibration of the prediction set $C_{1-\alpha}$ at each time step---, which may not be desirable in applications.

The main message of this work is that on many occasions there is no need to introduce specific CP methods for non-exchangeable data. We prove that in such cases split CP possesses the marginal and empirical guarantees above, up to the addition of a slightly larger penalty term $\eta$ in (\ref{eq:marginaliid}) and (\ref{eq:empiricaliid}). These guarantees hold in finite samples and make no underlying assumptions on model consistency. While the penalty depends on the nature of the non-exchangeability---more specifically, on decoupling and concentration properties of the process---, we show that in practice the effect is small even for moderately dependent data, and that increasing the calibration set size is a viable corrective. Importantly, split CP is computationally simple, avoiding intensive routines such as bootstrapping, ensembling or blocking. Finally, the method is exactly the same as the one used for the iid data and attests to its robustness, which is essential to ensure its validity in practical settings.

For example, Figure~\ref{fig:intro} shows how split CP's marginal coverage (\ref{eq:marginal_coverage_guarantee}), calculated through $10\,000$ simulations, behaves for an $\text{AR}(1)$ time series and three different underlying models. The data-generating mechanism is given by $Y_t = \lambda Y_{t-1}+\varepsilon_t$, $t\in\N$, $\lambda\in\R$ and $\varepsilon_t \sim N(0,1)$ independently. Split conformal quantile regression \citep{romano2019conformalized} is employed to achieve a nominal level of $90\%$, with models trained on 11 lags (i.e. $X_t=(Y_{t-j})_{j=1}^{11}$) to predict the next element in the sequence. The $x$-axis is indexed by $\lambda$, which can be interpreted as a level of dependence in the data. Unless the dependence is very high, split CP still has adequate coverage: autoregressive coefficients up to $\lambda=0.99$ achieve coverage higher than $89\%$. Significant losses of coverage only happen when $\lambda \geq 0.999$.

\begin{figure}[t]
    \centering
    \includegraphics[width=\textwidth]{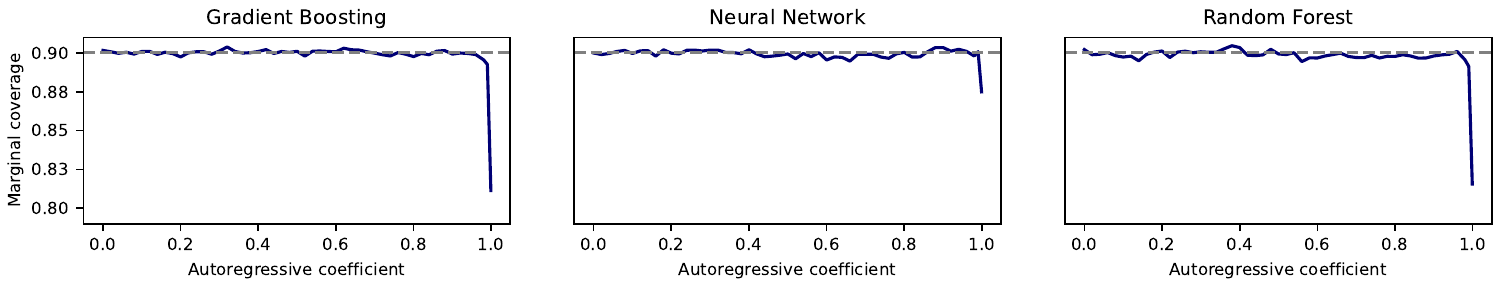}
    \caption{Marginal coverage for $\text{AR}(1)$ process (solid) and nominally prescribed iid level of 90\% (dashed) for different values of the autoregressive coefficient and three different models. Split CP holds well even under moderate dependence and undercoverage only happens at very high levels.}
    \label{fig:intro}
\end{figure}

Our main contributions are as follows:
\begin{itemize}
    \item We show coverage guarantees from split CP can be extended to large classes of dependent processes through the addition of a small coverage penalty;
    \item We do so by introducing a novel mathematical framework that is based on concentration inequalities and data decoupling properties rather than exchangeability;
    \item We present many generalizations that fit our framework, including both marginal, empirical and conditional coverage guarantees, as well as extensions to non-split CP methods and non-stationary data;
    \item We explicitly consider the broad class of stationary $\beta$-mixing processes (which includes, e.g., hidden Markov models, ARMA models and Markov chains), and show their empirical coverage bounds can match the order of the bounds under exchangeability;
    \item We conduct experiments that show split CP's success in generating prediction intervals in real and synthetic data, highlighting its advantages in terms of coverage, interval size, simplicity and speed compared to many recent algorithms, such as \citet{gibbs2024conformal, xu2021dynamic,xu2023conformal, barber2023conformal}.
\end{itemize}

Proofs are postponed to the appendices.

\section{Related Works}

There have been many new proposals in extending CP to non-exchangeable data. \citet{barber2023conformal} focus on distributional drift. They bound the \emph{coverage gap}---i.e., the difference between nominal and actual coverage levels---by a measure of deviation from exchangeability, which may be quite large for the time series or spatiotemporal data we deal with. \citet{gibbs2021adaptive,gibbs2024conformal} consider an online method where there is no calibration set and the quantile $\widehat{q}_{1-\alpha}$ is tuned online. They obtain very strong empirical coverage guarantees, but have no marginal coverage guarantees, and the online aspect of their method, which requires updates at every step, may be undesirable, e.g., in deployed applications. Similar comments apply to extensions and variants of this approach in \citet{feldman2022risk} and \citet{zaffran2022adaptive}.

\citet{chernozhukov2018blocks} design prediction sets for time series data via a block-permutation method. Their Theorem 2 gives approximate guarantees in a setting including $\alpha$-mixing processes, which is in principle weaker than our $\beta$-mixing assumption. However, this mixing condition interacts in a complicated way with the block structure (cf. Lemma 1 in \citealt{chernozhukov2018blocks}), which is a complicated hyperparameter that increases the computational cost. In other work, the same authors \citep{chernozhukov2021distributional} achieve stronger conditional guarantees for a simpler method than in \citet{chernozhukov2018blocks}, at the cost of much stronger assumptions. That is, they require that population objects be learned asymptotically from the data, which can be undesirable in real-world applications where one may want to be agnostic about the quality of the model. Similar comments apply to \citet{xu2021dynamic,xu2023conformal}, which have additional hyperparameters and computationally demanding algorithms.

In contrast, our framework for split CP holds for a broad range of non-exchangeable data; retains finite-sample marginal, empirical and conditional guarantees; does not require online modifications; has no hyperparameters and is both simpler and orders of magnitude faster than these alternatives. These advantages can be seen in the experiments in Section \ref{sec:empirical_studies}.

\section{Theoretical Results}

We begin this section by introducing the notation for the theoretical results. Denote by $(X_i,Y_i)_{i=1}^n$ a sample of $n$ random covariate/response pairs with stationary marginals. The pairs $(X_i,Y_i)$ take values in $\sX\times \sY$, where $\sX$ and $\sY$ are measurable spaces. An additional random pair $(X_*,Y_*)$ with values in $\sX\times \sY$, independent from the sample $(X_i,Y_i)_{i=1}^n$, will also be considered, and we assume $(X_i,Y_i)\sim (X_*,Y_*)$ for all $i\in[n]$, where $[n]\coloneqq\{1,\ldots,n\}$.

The data indices can be partitioned $[n]=I_{\train}\sqcup I_{\cal}\sqcup I_{\test}$, where $n=n_{\train}+n_{\cal} + n_{\test}$ and $I_{\train} \coloneqq [n_{\train}]$ corresponds to the training data, $I_{\cal} \coloneqq [n_{\train}+n_{\cal}]\backslash [n_{\train}]$ corresponds to calibration data, and $I_{\test} \coloneqq [n]\backslash [n_{\train}+n_{\cal}]$ corresponds to test data.

For any function $s \colon (\sX\times \sY)^{n_{\train}+1}\to \R$ and $(x,y)\in\sX\times \sY$, denote a nonconformity score as
\begin{equation*}
    \widehat{s}_{\train}(x,y) \coloneqq s((X_i,Y_i)_{i\in I_{\train}},(x,y)),
\end{equation*}
corresponding to the values of $s$ when the first $n_{\train}$ pairs in the input are the training data. Intuitively, the role of $\widehat{s}_{\train}$ is to measure how discrepant a prediction based on $x_i$ is compared to the true $y_i$; e.g., $\widehat{s}_{\train}(x,y)=|y-\widehat{\mu}(x)|$, where $\widehat{\mu}$ is some regression model trained on ${(X_i, Y_i)}_{i\in I_{\train}}$. Several choices have been proposed in the literature \citep{lei2018distribution, hechtlinger2019cautious, romano2019conformalized, angelopoulos2021uncertainty}.

Given $\phi\in [0,1)$, let $\widehat{q}_{\phi,\cal}$ denote the empirical $\phi$-quantile of $\widehat{s}_{\train}(X_i,Y_i)$ over $I_{\cal}$; i.e.,
\begin{equation} \label{eq:quantile}
    \widehat{q}_{\phi,\cal} \coloneqq \inf\Biggl\{t\in\R\,:\, \frac{1}{n_{\cal}}\sum_{i\in I_{\cal}}\Ind{\widehat{s}_{\train}(X_i,Y_i)\leq t}\geq \phi\Biggr\}.
\end{equation}
For $x\in\sX$, the prediction sets are then defined via $C_{\phi}(x) \coloneqq \{y\in\sY\,:\, \widehat{s}_{\train}(x,y)\leq \widehat{q}_{\phi,\cal}\}$. Also, given a measurable function $q\colon(\sX\times \sY)^{n_{\train}}\to \R$ (which can be thought of as a quantile), define $q_{\train} \coloneqq q((X_i,Y_i)_{i\in I_{\train}})$ and the probability
\begin{equation} \label{eq:p_q_train}
   {P}_{q,\train} \coloneqq \P[\widehat{s}_{\train}(X_*,Y_*)\leq q_{\train}\mid (X_i,Y_i)_{i\in I_{\train}}].
\end{equation}

\subsection{Marginal and Empirical Guarantees}\label{sec:general_framework}

This subsection details how marginal and empirical guarantees (\ref{eq:marginaliid}) and (\ref{eq:empiricaliid}) can be extended when the data is not exchangeable. Some basic assumptions are needed, though they are satisfied by large classes of processes. Section \ref{sec:stationary_beta_mixing_data} shows that is the case for stationary $\beta$-mixing data, but the general framework developed here generalizes to other processes.

First, it is necessary to have some form of concentration over the calibration data, as well as a degree of decoupling over the test data. We assume there exist $\varepsilon_{\cal}\in (0,1)$ and $\delta_{\cal}\in (0,1)$ such that
    \begin{equation}\label{eq:hypothesis-concentration_over_calibration}
        \P\Biggl[\Biggl|\frac{1}{n_{\cal}}\sum_{i\in I_{\cal}}\Ind{\widehat{s}_{\train}(X_i,Y_i)\leq q_{\train}} - P_{q,\train}\Biggr|\leq \varepsilon_{\cal}\Biggr] \geq 1-\delta_{\cal},
    \end{equation}
where $P_{q,\train}$ is defined as in (\ref{eq:p_q_train}).
Intuitively, this condition requires that the empirical and population cumulative distribution functions of $\widehat{s}_{\train}(X,Y)$ are close over calibration data. A key point here, however, is that this closeness should hold even when the cdf is computed over a point depending on training data.

Further, we assume that there exists $\varepsilon_{\train}$ such that, for $i \in I_{\test}$,
    \begin{equation}\label{eq:hypothesis-decoupling_over_test_data}
    \left|\P[\widehat{s}_{\train}(X_i,Y_i)\leq q_{\train}] - \P[\widehat{s}_{\train}(X_*,Y_*)\leq q_{\train}] \right|\leq \varepsilon_{\train}.
    \end{equation}
This means that $(X_i,Y_i)$ for $i\in I_{\test}$ essentially behaves like $(X_*,Y_*)$, i.e., a data point that is independent of training data.

Conditions like (\ref{eq:hypothesis-concentration_over_calibration}) and (\ref{eq:hypothesis-decoupling_over_test_data}) (with small error parameters) hold for many types of dependent data, such as strictly stationary processes whose ``memory'' is not too strong. Examples include finite Markov chains, hidden Markov chains and ARMA-type processes (both linear and nonlinear). A broader class of such processes can be described under so-called $\beta$-mixing conditions, which we discuss in detail in Section \ref{sec:stationary_beta_mixing_data}. Under these conditions, the usual marginal coverage guarantees can be recovered for split CP.

\begin{theorem}[Marginal coverage over test data]\label{theorem:marginal_coverage_test_data}
    Given $\alpha\in(0,1)$ and $\delta_{\cal}>0$, if conditions (\ref{eq:hypothesis-concentration_over_calibration}) and (\ref{eq:hypothesis-decoupling_over_test_data}) hold, then, for all $i\in I_{\test}$:
    \begin{equation} \label{eq:marginal_coverage_guarantee}
        \P[Y_i\in C_{1-\alpha}(X_i)]\geq 1-\alpha - \varepsilon_{\cal} - \delta_{\cal}-\varepsilon_{\train}.
    \end{equation}
    Additionally, if $\widehat{s}_{\train}(X_*,Y_*)$ almost surely has a continuous distribution conditionally on the training data, then
    \begin{equation*}
       \left|\P[Y_i\in C_{1-\alpha}(X_i)]- (1-\alpha)\right|\leq  \varepsilon_{\cal} + \delta_{\cal} + \varepsilon_{\train}.
    \end{equation*}
\end{theorem}

To also guarantee empirical coverage, suppose that instead of the decoupling assumption (\ref{eq:hypothesis-decoupling_over_test_data}), there exists concentration of the empirical cumulative distribution function of the nonconformity score over the test data, that is, there exist $\varepsilon_{\test},\delta_{\test}\in(0,1)$ such that
    \begin{equation} \label{eq:hypothesis-concentration_of_empirical_cdf}
        \P\left[ \Biggl|P_{q,\train}- \frac{1}{n_{\test}} \sum_{i\in I_{\test}}\Ind{\widehat{s}_{\train}(X_i,Y_i)\leq q_{\train }}\Biggr|\leq \varepsilon_{\test}\right] \geq 1-\delta_{\test}.
    \end{equation}

\begin{theorem}[Empirical coverage over test data]\label{theorem:empirical_coverage_test_data}
    Given $\alpha\in(0,1)$, $\delta_{\cal}>0$ and $\delta_{\test}>0$, if (\ref{eq:hypothesis-concentration_over_calibration}) and (\ref{eq:hypothesis-concentration_of_empirical_cdf}) hold, then:
\begin{equation*}
    \P\left[\frac{1}{n_{\test}} \sum_{i \in I_{\test}} \Ind{Y_i\in C_{1-\alpha}(X_i)} \geq 1 - \alpha - \eta\right] \geq 1 - \delta_{\cal}-\delta_{\test},
\end{equation*}
where $\eta = \varepsilon_{\cal}+\varepsilon_{\test}.$ Additionally, if $\widehat{s}_{\train}(X_*,Y_*)$ almost surely has a continuous distribution conditionally on the training data, then:
\begin{equation*}
    \P\left[\Biggl|\frac{1}{n_{\test}} \sum_{i \in I_{\test}} \Ind{Y_i\in C_{1-\alpha}(X_i)} - (1 - \alpha)\Biggr| \leq  \eta\right] \geq 1 - 2\delta_{\cal}-2\delta_{\test}.
\end{equation*}
\end{theorem}

While the purpose of the above theorems is to extend split conformal guarantees to non-exchangeable data, they also readily apply to the iid case. Indeed, it is straightforward to show that, in such case, it suffices to take $\varepsilon_{\cal}=\sqrt{(2 n_{\cal})^{-1}\log(2/\delta_{\cal})}$ and $\varepsilon_{\test}=\sqrt{(2n_{\test})^{-1} \log(2/\delta_{\test})}$.

\subsection{Conditional Guarantees}\label{sec:conditional_guarantees}

Obtaining a conditional version of (\ref{eq:marginaliid}) and (\ref{eq:empiricaliid}) is of interest in many cases. \citet{barber2020limits} prove that coverage is not generally attainable, even for iid data. On the positive side, they show they can be achieved when conditioning on sets of finite VC dimension that are not too small. Our goal is to show similar guarantees for split conformal prediction under non-exchangeable data.

First, conditional versions of assumptions (\ref{eq:hypothesis-concentration_over_calibration}) and (\ref{eq:hypothesis-decoupling_over_test_data}) are needed. For concentration over the calibration data, suppose there exist $\delta_{\cal}$ and $\varepsilon_{\cal}\in(0,1)$ such that, for ${P}_{q,\train}(A) \coloneqq \P[\widehat{s}_{\train}(X_*,Y_*)\leq q_{\train}\mid (X_i,Y_i)_{i\in I_{\train}},X_*\in A]$ and $n_{\cal}(A)\coloneqq \#\{i\in I_{\cal}\,:\,X_i\in A\}$,
    \begin{equation} \label{eq:hypothesis-concentration_over_calibration-conditional}
        \P\Biggl[\sup_{A\in\sA}\Biggl|\frac{1}{\max\{n_{\cal}(A),1\}}\sum_{i\in I_{\cal}(A)}\Ind{\widehat{s}_{\train}(X_i,Y_i)\leq q_{\train }} - P_{q,\train}(A)\Biggr|\leq \varepsilon_{\cal}\Biggr] \geq 1-\delta_{\cal}.
    \end{equation}
For a conditional version of marginal decoupling, assume there exists $\varepsilon_{\train}\in(0,1)$ such that
\begin{equation} \label{eq:hypothesis-decoupling_over_test_data-conditional}
    \left|\P[\widehat{s}_{\train}(X_k,Y_k)\leq q_{\train}\mid X_k\in A] - \P[\widehat{s}_{\train}(X_*,Y_*)\leq q_{\train}\mid X_*\in A]\right|\leq \varepsilon_{\train}.
\end{equation}
These conditions suffice for conditional marginal coverage.

\begin{theorem}[Conditional coverage over test data]\label{theorem:marginal_coverage_test_data-conditional}
Given $\alpha\in(0,1)$ and $\delta_{\cal}>0$, if (\ref{eq:hypothesis-concentration_over_calibration-conditional}) and (\ref{eq:hypothesis-decoupling_over_test_data-conditional}) hold, then, for each $A\in \sA \subset \sX$  and any $i\in I_{\test}$:
    \begin{equation}\label{eq:conditional_coverage_guarantee}
        \P[Y_i\in C_{1-\alpha}(X_i;A)\mid X_i\in A]\geq 1-\alpha - \varepsilon_{\cal} - \delta_{\cal}-\varepsilon_{\train}.
    \end{equation}
Additionally, if $\widehat{s}_{\train}(X_*,Y_*)$ almost surely has a continuous distribution conditionally on the training data, then:
    \[
        \left|\P[Y_i\in C_{1-\alpha}(X_i;A)\mid X_i\in A]- (1-\alpha)\right|\leq  \varepsilon_{\cal} + \delta_{\cal} + \varepsilon_{\train}.
    \]
\end{theorem}

Appendix~\ref{appendix:conditional_guarantees} includes a conditional version of the empirical coverage guarantee.

\subsection{Application to Stationary \texorpdfstring{$\beta$}{Beta}-Mixing Data} \label{sec:stationary_beta_mixing_data}

We now apply the framework from Section \ref{sec:general_framework} to the class of stationary $\beta$-mixing processes. This class of non-exchangeable data is broad enough to cover many important applications, such as hidden Markov models and Markov chains \citep{doukhan2012mixing} as well as ARMA and GARCH models \citep{carrasco2002mixing, mokkadem1988mixing}, while still providing explicit error terms in the bounds of Theorems \ref{theorem:marginal_coverage_test_data} and \ref{theorem:empirical_coverage_test_data}.

Recall a sequence of random elements $\{Z_t\}_{t=-\infty}^{\infty}$ of a measurable space $\sZ$ is stationary if its finite-dimensional distributions are time-invariant; that is, for any $t \in \Z$ and $m, k \in \N$,
    \begin{equation*}
    Z_{t:(t+m)} = (Z_t, \ldots, Z_{t+m}) \stackrel{d}{=} (Z_{t+k}, \ldots, Z_{t+m+k}) = Z_{(t+k):(t+m+k)}.
    \end{equation*}
Furthermore, for a stationary stochastic process $\{Z_{t}\}_{t=-\infty}^{\infty}$ and index $a \in \N$, the $\beta$-mixing coefficient of the process at $a$ is defined as
    \begin{equation*}
        \beta(a) = \lVert\P_{-\infty:0, a:\infty} - \P_{-\infty:0}\otimes\P_{a:\infty}\rVert_{\TV},
    \end{equation*}
where $\|\cdot\|_{\TV}$ denotes the total variation norm, and $\P_{-\infty:0, a:\infty}$ is the joint distribution of the blocks $(Z_{-\infty:0}, Z_{a:\infty})$. The process is said to be $\beta$-mixing if $\beta(a) \to 0$ when $a \to \infty$.

The $\beta$-mixing condition allows us to replace independence with asymptotic independence and still retain some important concentration results. In particular, the so-called Blocking Technique \citep{yu1994rates, mohri2010stability, kuznetsov2017generalization} allows one to compare a $\beta$-mixing process with another process made of independent blocks. The results below generally follow from combining the Blocking Technique with decoupling arguments and Bernstein's concentration inequality.
Crucially, the guarantees might depend on block sizes, but split CP itself does not. This is an advantage over CP variants \citep{chernozhukov2018blocks}.

The sets of parameters we optimize over are defined as follows:
\begin{align*}
    F_{\cal} &= \{(a,m,r)\in \N_{>0}^3\ :\ 2ma = n_{\cal}-r+1,\ \delta_{\cal}>4(m-1)\beta(a)+\beta(r)\},\\
    F_{\test} &= \{(a,m,s)\in \N_{>0}^2\times \N\ :\ 2ma = n_{\test}-s,\ \delta_{\test}>4(m-1)\beta(a)+\beta(n_{\cal})\}.
\end{align*}

These two sets correspond to block size choices in the calibration and test sets, respectively. For the calibration set, define the error term as follows:
\begin{align}\label{eq:eps_cal_beta}
        \varepsilon_{\cal} \coloneqq \inf_{(a,m,r)\in F_{\cal}}&\Biggl\{ \widetilde{\sigma}(a)\sqrt{\frac{4}{n_{\cal}-r+1} \log\biggl(\frac{4}{\delta_{\cal}-4(m-1)\beta(a)-\beta(r)}\biggr)}\\
            &\quad + \frac{1}{3m}\log\biggl(\frac{4}{\delta_{\cal}-4(m-1)\beta(a)-\beta(r)}\biggr) + \frac{r-1}{n_{\cal}}\Biggr\}\nonumber,
\end{align}
where $\widetilde{\sigma}(a)=\sqrt{\frac{1}{4} +\frac{2}{a}\sum_{j=1}^{a-1}(a-j)\beta(j)}$. Similarly, we define the test error correction factor for a stationary $\beta$-mixing process as
\begin{align}\label{eps_test_beta}
        \varepsilon_{\test}= \inf_{(a,m,s)\in F_{\test}}&\Biggl\{ \widetilde{\sigma}(a)\sqrt{\frac{4}{n_{\test}} \log\biggl(\frac{4}{\delta_{\test}-4(m-1)\beta(a)-\beta(n_{\cal})}\biggr)}\\
         &\quad + \frac{1}{3m}\log\biggl(\frac{4}{\delta_{\test}-4(m-1)\beta(a)-\beta(n_{\cal})}\biggr) + \frac{s}{n_{\rm test}}\Biggr\}\nonumber.
\end{align}

We emphasize that this optimization of block sizes plays a role exclusively on the coverage guarantees below; the split CP algorithm itself remains unchanged.

With $\varepsilon_{\cal}$ as above, Theorem \ref{theorem:marginal_coverage_test_data} yields the following result for stationary $\beta$-mixing processes:

\begin{theorem}[Marginal coverage: stationary $\beta$-mixing processes]\label{theorem:general_framework_1_beta_mixing}
    Suppose  the sample $(X_i, Y_i)_{i=1}^{n}$ is stationary $\beta$-mixing. Then given $\alpha\in(0,1)$ and $\delta_{\cal}>0$, for $i\in I_{\test}$,
    \begin{equation*}
        \P\left[Y_i\in C_{1-\alpha}(X_i)\right] \geq 1-\alpha-\eta,
    \end{equation*}
with $\eta = \varepsilon_{\cal} + \varepsilon_{\train} +\delta_{\cal},$ where $\varepsilon_{\cal}$ is as  in (\ref{eq:eps_cal_beta}) and $\varepsilon_{\train} = \beta(i-n_{\train}).$
Additionally, if $\widehat{s}_{\train}(X_*,Y_*)$ almost surely has a continuous distribution conditionally on the training data:
\[|\P[Y_i\in C_{1-\alpha}(X_i)]- (1-\alpha)|\leq  \eta.\]
\end{theorem}

Under certain assumptions over the dependence of the processes, the stationary $\beta$-mixing bounds given by (\ref{eq:eps_cal_beta}) are of the same asymptotic order as the corresponding iid bounds. Indeed, if $\beta(k)\leq k^{-b}$ and $\delta\geq n_{\cal}^{-c}$ for $b>1,c>0$, with $1+2c<b$, as long as $m= o(n_{\cal}^{(b-c)/(b+1)})$ and $\sqrt{n_{\cal}\log(n_{\cal})}= o\left(m\right)$, the bounds are of the same order. This is satisfied, for example, if $m=n_{\cal}^{\lambda}$, $a= n_{\cal}^{1-\lambda}/2$ with $1/2<\lambda<(b-c)/(b+1)$.

Additionally, with $\varepsilon_{\test}$ as above, Theorem \ref{theorem:empirical_coverage_test_data} yields the following:
\begin{theorem}[Empirical coverage: stationary $\beta$-mixing processes]\label{theorem:general_framework_2_beta_mixing}
    Suppose the sample $(X_i, Y_i)_{i=1}^{n}$  is  stationary $\beta$-mixing. Then given $\alpha\in(0,1)$, $\delta_{\cal}>0$ and $\delta_{\test}>0$
   \begin{equation*}
   \P\left[\frac{1}{n_{\test}} \sum_{i \in I_{\test}} \Ind{Y_i\in C_{1-\alpha}(X_i)} \geq 1 - \alpha - \eta\right] \geq 1 - \delta_{\cal}-\delta_{\test},
\end{equation*}
with $\eta = \varepsilon_{\cal}+\varepsilon_{\test}$, and $\varepsilon_{\cal}$ and $\varepsilon_{\test}$ defined in (\ref{eq:eps_cal_beta}) and (\ref{eps_test_beta}). Additionally, if $\widehat{s}_{\train}(X_*,Y_*)$ almost surely has a continuous distribution conditionally on the training data, then:
\begin{equation*}
    \P\left[\Biggl|\frac{1}{n_{\test}} \sum_{i \in I_{\test}} \Ind{Y_i\in C_{1-\alpha}(X_i)} - (1 - \alpha)\Biggr| \leq  \eta\right] \geq 1 - 2\delta_{\cal}-2\delta_{\test}.
\end{equation*}
\end{theorem}

We note in passing that the expression in (\ref{eq:eps_cal_beta}) follows from a stationary $\beta$-mixing version of Bernstein's inequality which might be of independent interest. See Appendix~\ref{appendix:stationary_beta_mixing_data} for the proof and for conditional extensions of the above results.

\subsection{Extensions to Non-Split CP Methods and Non-Stationary Data}\label{sec:extensions}

The results obtained in previous subsections also apply when the data is non-stationary. For brevity, we focus on marginal coverage. Our analysis is partly inspired by the recent work of \citet{barber2023conformal}.

Replace the pair $(X_*,Y_*)$ with an auxiliary process $(X_{*,i},Y_{*,i})_{i\in [n]}$ that is an independent copy of the original data $(X_i,Y_i)_{i\in [n]}$. Let $N_{\cal}$ be a random number, uniformly distributed over $I_{\cal}$, independently of the problem data and auxiliary process. For $j \in I_{\test}$, the quantity $\delta^{(j)} \coloneqq \| \mathrm{Law}(X_{j},Y_{j}) - \mathrm{Law}(X_{N_{\cal}},Y_{N_{\cal}})\|_{\rm TV}$, for $j\in I_{\test}$,
measures how far the distribution of $(X_j,Y_j)$ is to that of a randomly chosen point in the calibration data set (i.e., a measure of distributional drift).

For marginal coverage, we take the random variable $q_{\train}$ as before, but replace (\ref{eq:p_q_train}) by a time-inhomogeneous version, that is, $P_{q,\train}^{(j)} \coloneqq \P[\widehat{s}_{\train}(X_{*,j},Y_{*,j})\leq q_{\train}\mid (X_i,Y_i)_{i\in I_{\train}}]$.
Furthermore, (\ref{eq:hypothesis-concentration_over_calibration}) and (\ref{eq:hypothesis-decoupling_over_test_data}) are also replaced with time-inhomogeneous versions:
\begin{equation}\label{eq:concentrationnonstat}
        \P\Biggl[\Biggl|\frac{1}{n_{\cal}}\sum_{i\in I_{\cal}}(\Ind{\widehat{s}_{\train}(X_i,Y_i)\leq q_{\train}} - P^{(i)}_{q,\train}) \Biggr|\leq \varepsilon_{\cal}\Biggr] \geq 1-\delta_{\cal},
\end{equation}
and, for $j \in I_{\test}$,
\begin{equation}\label{eq:decouplingnonstat}
        \lvert\P[\widehat{s}_{\train}(X_j,Y_j)\leq q_{\train}] - \P[\widehat{s}_{\train}(X_{*,j},Y_{*,j})\leq q_{\train}]\rvert\leq \varepsilon_{\train}.
\end{equation}

\begin{theorem}[Marginal coverage for split CP on non-stationary data]\label{theorem:marginal_nonstat} Given a level $\alpha\in (0,1)$, if (\ref{eq:concentrationnonstat}) and (\ref{eq:decouplingnonstat}) hold, then for all $j\in I_{\test}$:
\[\P[\widehat{s}_{\train}(X_{j},Y_{j})\leq \widehat{q}_{1-\alpha,\cal}]\geq 1-\alpha-\eta - \delta^{(j)},\]
where $\eta=\varepsilon_{\cal}+\delta_{\cal}+ \varepsilon_{\train}$ is as in Theorem \ref{theorem:marginal_coverage_test_data}.\end{theorem} In particular, we recover Theorem \ref{theorem:marginal_coverage_test_data} up to an error depending on how much distributional drift there is between $j$ and the calibration set. This is similar to the main result in \citet{barber2023conformal}, except that there the authors consider weighted calibration sets.

Finally, our framework also extends to other popular methods in the literature that are not based on split CP, such as rank-one-out (ROO) conformal prediction \citep{lei2018distribution} and risk-controlling prediction sets (RCPS) \citep{bates2021rcps}. ROO calibrates the quantiles used for each test data point by using the remaining test points, so it is different from split CP. RCPS gives a general CP methodology that applies in a variety of settings, including regression, multiclass classification and image segmentation. Importantly, RCPS does not involve nonconformity scores but the construction of nested sets. Details for both of these extensions to non-exchangeable data are given in Appendix~\ref{appendix:extensions}.

\section{Experiments}\label{sec:empirical_studies}

This section studies split CP's empirical performance in several numerical experiments. The first one uses real spatiotemporal climate data to compare split CP's coverage and average interval size to recent alternative conformal algorithms due to \citet{barber2023conformal} and \citet{gibbs2024conformal}. The second experiment benchmarks split CP on synthetic $\beta$-mixing data against a popular conformal approach for time-series \citep{xu2021dynamic,xu2023conformal}. In both cases, split conformal is used with absolute residuals as nonconformity score. The third example involves a hidden Markov model in which the bounds can be calculated explicitly, while the last shows that split CP's marginal and conditional guarantees work with real financial data, even when exchangeability is clearly violated. In these two final examples, as in the $\text{AR}(1)$ experiment (cf. Figure~\ref{fig:intro}), we employ split conformal quantile regression \citep{romano2019conformalized}. The experiments were conducted on a server with 774 GB of RAM and two Intel Xeon Platinum 8354H processors, totalling 8 physical cores and 288 threads. All conformal implementations made use of multiprocessing for better performance. Further implementation details are discussed in Appendix~\ref{appendix:empirical_studies}. Code to reproduce all the figures and tables is available at \url{https://github.com/jv-rv/split-conformal-nonexchangeable}.

\textit{Example 1 (Spatiotemporal climate data)} Let $Y_{t,l} \in \R$ be the 14-day forward rolling average temperature in degrees Celsius, with $t$ representing the first measurement day included in the average and $l \in L$ a location on Earth defined by latitudes and longitudes discretized in a $1.5\degree \times 1.5\degree$ grid totalling 7512 distinct locations. For all grid locations, measurements from the start of 1979 through the end of 2022 were retrieved via the National Oceanic and Atmospheric Administration \citep{noaa-ncep-cpc-temperature-daily}.

It will be notationally convenient to split the dates in year, month and day, so we consider instead the equivalent representation $Y_{y,m,d,l}$. A standard procedure in the meteorological literature to predict temperatures for a given date $t$ and location $l$, called climatology, is to average the temperatures observed on the same day, month and location over the previous 30 years \citep{hwang2019improving,arguez2012noaa}. More precisely, we take $X_{y,m,d,l} \coloneqq (Y_{y-j,m,d,l})_{j=1}^{30}$ as features and $\hat{\mu}(\cdot) = \Avg(\cdot)$ as the prediction model. Since 30 years (1979--2008) of temperature data are needed for the features, our response series starts on 2009.

Note that the model $\hat{\mu}$ is fixed \emph{a priori} and needs no training data, so $I_{\train} = \emptyset$. For a new test point $Y_{y_*,m_*,d_*,l_*}$, the calibration set $I_{\cal}$ is taken as all the previous observations for that same day and month and the absolute residual is used as nonconformity score function. Thus, the set of nonconformity scores evaluated on calibration data is given by
\begin{equation}
\{ \lvert Y_{y,m,d,l} - \hat{\mu}(X_{y,m,d,l}) \rvert : y < y_*,\; m = m_*,\; d = d_*,\; l \in L\}.
\label{eq:spatiotempora-calibration-residuals}
\end{equation}

The intuition behind this choice is that temperature measurements for a given location, day and month can be reasonably modeled as $\beta$-mixing and stationary over the years \citep{gardner2006cyclostationarity}.

Split conformal prediction is efficient in generating prediction intervals in this scenario. At any given date, the calibration set for all 7512 geographic points is the same and taking the quantile of (\ref{eq:spatiotempora-calibration-residuals}) allows us to build thousands of prediction intervals at once. We follow the online sequential procedure through time implied by (\ref{eq:spatiotempora-calibration-residuals}), with no additional overhead due to the spatial dependence. Figure~\ref{fig:splitcp-climatology-pred_intervals} shows the temperature measurements in the center, from which the spatial dependence becomes evident. The prediction intervals are represented by their minimum (left) and maximum (right). Coverage is generally attained and intervals are narrow enough to be useful.

\begin{figure}[t]
    \centering
    \includegraphics[width=\textwidth]{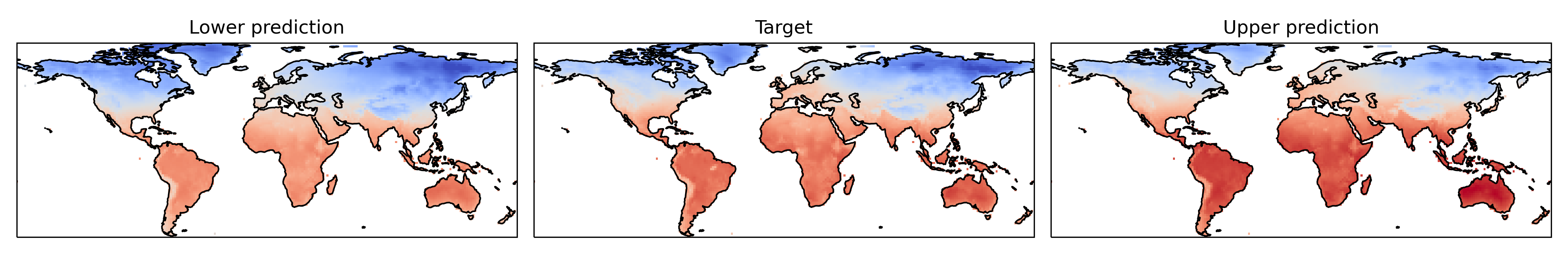}
    \caption{Average temperatures around the globe on 2022-12-31 (center), split CP's lower prediction (left) and upper prediction (right). Darker blues (reds) indicate lower (higher) temperatures. Approximately 92.6\% of the 7512 grid points are inside the prediction interval, which is close to the 90\% prescribed coverage. Averaging over all days, coverage is of 90.9\%}
    \label{fig:splitcp-climatology-pred_intervals}
\end{figure}

We compare the standard and widespread split CP method \citep{papadopoulos2002inductive,vovk2005algorithmic,lei2018distribution} against recent developments tailored to non-exchangeable data, namely NexCP due to \citet{barber2023conformal} and dynamically-tuned adaptive conformal inference (DtACI) due to \citet{gibbs2024conformal}, which we briefly describe next.

\citet{barber2023conformal} proposed variations of conformal prediction to deal with non-exchangeability. We focus on the split version of their NexCP algorithm, that is identical to split CP apart from the fact that quantiles are weighted to give more importance to more dependent residuals. In the original paper, this novel method is evaluated on time series data and shows improvement on coverage for data presenting distribution drift. For the non-exchangeable violations studied in our work, however, one would expect standard split CP to work well enough.

Using NexCP for spatiotemporal data requires a weighting scheme that takes into account both space and time. We deal with the time dimension as in the original paper via $\rho_{\mathrm{time}}^{y_* - y}$ for some decay parameter $\rho_{\mathrm{time}} \in (0, 1]$, giving less weight to observations from the distant past. In a similar vein, we employ the exponential decay $\rho_{\mathrm{space}}^{d(l_*, l)}$ for $\rho_{\mathrm{space}} \in (0, 1]$ and $d(l_*, l)$ a distance between the test point location $l_*$ and the residual location $l$. A good approximation for the distance of two points on Earth defined in terms of latitudes and longitudes is achieved by assuming a spherical surface and making use of the haversine formula. We employ this procedure so that $d\colon \R^2 \times \R^2 \to [0, \pi]$ is the angular distance between $l_*$ and $l$. The final weights are given by $\rho_{\mathrm{time}}^{y_* - y} \cdot \rho_{\mathrm{space}}^{d(l_*, l)}$. Calculating a weighted quantile has average time complexity $O(n_{\cal} \log n_{\cal})$ in contrast to $O(n_{\cal})$ for the unweighted case. Moreover, NexCP must calculate a different quantile for each one of the 7512 geographic points, while Split CP makes one single calculation per time step. Therefore, one would expect NexCP to be much slower than Split CP for spatiotemporal data, even when parallelized (cf. Table~\ref{table:benchmarks}).

\citet{gibbs2024conformal} introduced dynamically-tuned adaptive conformal inference (DtACI) as an online learning method for constructing prediction intervals with a target coverage. As usual in the online setting, it is assumed that one new test point arrives at each time step. In the case of spatiotemporal data, with all grid points arriving at once, a reasonable strategy is to consider many individual time series. Indeed, this is how Covid-19 predictions are dealt with in the original paper and how we proceed.

\begin{table}[t]
\footnotesize
\centering
\begin{tabular}{ccccccc}

\toprule
\multirow{2}{*}{Method}      & \multirow{2}{*}{Hyperparams}                                             & \multicolumn{3}{c}{Coverage}     & \multirow{2}{*}{Interval size} & \multirow{2}{*}{Avg time} \\
\cmidrule(lr){3-5}           &                                                                          & Avg     & SD (time) & SD (space) &    (\degree C)                 &    (min)                                 \\
\midrule
Split CP                     & None                                                                     & 90.9\%  & 0.033     & 0.130      & 9.605                          & 2.85                          \\
\midrule
\multirow{6}{*}{NexCP}       & $\rho_{\mathrm{time}} = \rho_{\mathrm{space}} = 0.99$                    & 90.9\%  & 0.033     & 0.130      & 9.601                          & \multirow{6}{*}{1884.92}      \\
                             & $\rho_{\mathrm{time}} = \rho_{\mathrm{space}} = 0.9$                     & 90.8\%  & 0.033     & 0.131      & 9.554                          &                               \\
                             & $\rho_{\mathrm{time}} = \rho_{\mathrm{space}} = 0.7$                     & 90.6\%  & 0.033     & 0.132      & 9.446                          &                               \\
                             & $\rho_{\mathrm{time}} = \rho_{\mathrm{space}} = 0.5$                     & 90.3\%  & 0.035     & 0.133      & 9.360                          &                               \\
                             & $\rho_{\mathrm{time}} = \rho_{\mathrm{space}} = 0.3$                     & 90.0\%  & 0.038     & 0.134      & 9.289                          &                               \\
                             & $\rho_{\mathrm{time}} = \rho_{\mathrm{space}} = 0.1$                     & 89.5\%  & 0.042     & 0.136      & 9.205                          &                               \\
\midrule
\multirow{5}{*}{DtACI}        & $\lvert I \rvert = 1$, $\gamma \in \smash{(\tfrac{2^i}{1000})}_{i=0}^8$  & 91.9\%  & 0.034     & 0.096      & 9.658                          & \multirow{5}{*}{75.15}        \\
                             & $\lvert I \rvert = 3$, $\gamma \in \smash{(\tfrac{2^i}{1000})}_{i=0}^8$  & 92.1\%  & 0.034     & 0.090      & 9.779                          &                               \\
                             & $\lvert I \rvert = 5$, $\gamma \in \smash{(\tfrac{2^i}{1000})}_{i=0}^8$  & 92.1\%  & 0.034     & 0.090      & 9.806                          &                               \\
                             & $\lvert I \rvert = 7$, $\gamma \in \smash{(\tfrac{2^i}{1000})}_{i=0}^8$  & 92.1\%  & 0.034     & 0.090      & 9.814                          &                               \\
                             & $\lvert I \rvert = 9$, $\gamma \in \smash{(\tfrac{2^i}{1000})}_{i=0}^8$  & 92.2\%  & 0.034     & 0.091      & 9.815                          &                               \\
\bottomrule

\end{tabular}
    \caption{Comparison of conformal methods in terms of average running time, hyperparameters, average coverage, standard deviation of coverage over time and space, and average interval size. The prescribed level is 90\%. While all methods are close to the nominal level, split CP is simpler, has no hyperparameter and is orders of magnitude faster for the spatiotemporal data considered.\label{table:benchmarks}}
\end{table}

Table~\ref{table:benchmarks} summarizes the comparison between split CP, NexCP and DtACI. The main takeaway is that split CP---a much simpler method, with no hyperparameters, and orders of magnitude faster---is comparable to the benchmarks both in terms of coverage and interval size, working as expected for the spatiotemporal data. NexCP's weighting scheme was useful in reducing the average interval size while maintaining a coverage above the prescribed 90\% level, but the gains were small (less than 4\% for the optimal choice $\rho_{\mathrm{time}}$ and $\rho_{\mathrm{space}}$ equal to 0.3). DtACI reduced the coverage standard deviation over space, but generated larger intervals.

\citet{xu2021dynamic,xu2023conformal} developed ensemble batch prediction intervals (EnbPI) specifically as a conformal method for time-series data. However, it crucially relies on fitting an underlying prediction algorithm on bootstrapped samples of the training data. Recall that the climatology model uses no training data and simply takes the average of the features, so $I_{\train} = \emptyset$ and EnbPI cannot be applied. In order to compare split CP to this important recent development, we conducted the following experiment.

\textit{Example 2 (Hidden random walk on the cycle graph)}
Consider a $\beta$-mixing sequence generated from a random walk on the cycle graph. On any given state, there is a probability $b$ of moving backward, $f$ of moving forward and $s$ of staying still. The dependence increases for larger values of $s$, which makes it more likely to have repeating values. Intuitively, the number of vertices also influence dependence, as it should take a larger sample to account for a larger number of states to be visited. Indeed, for $r \in \N_{>0}$ the $\beta$-mixing coefficient $\beta(r)$ for a cycle of $v$ vertices decays at rate $e^{-r/v^2}$. A Gaussian noise with zero mean and variance of $10^{-6}$ is added to the resulting sequence to avoid draws.

\begin{figure}[tbp]
    \centering
    \includegraphics[width=.7\textwidth]{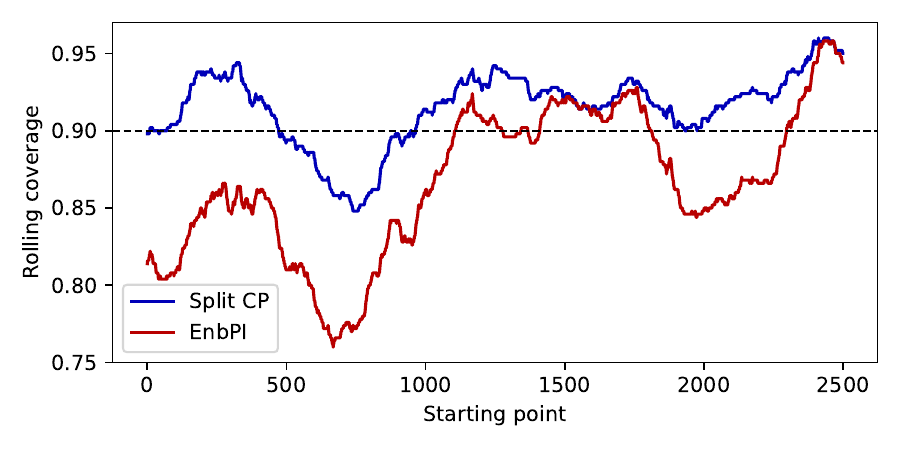}
    \caption{Split CP's rolling coverage for $\beta$-mixing data generated from a random walk on the cycle graph is close to the prescribed 90\% level, whereas EnbPI significantly undercovers at some points. Considering all test points, Split CP's coverage was of 91.4\% with an average interval size of 0.34; EnbPi's coverage was of 86.6\% with an average interval size of 0.15.}
    \label{fig:rolling_coverage-split_cp-enbpi}
\end{figure}

Figure~\ref{fig:rolling_coverage-split_cp-enbpi} compares the rolling coverage $\frac{1}{500} \sum_{i=j}^{499+j} \Ind{Y_i\in C_{1-\alpha}(X_i)}$ for $j\in\{1,\ldots,2501\}$ of split CP and EnbPI, both using a random forest trained on 11 lags to predict the next element in the sequence as model and $\alpha$ set to $0.1$. The implementation of EnbPI used was from \texttt{MAPIE} \citep{taquet2022mapie}. In this experiment, split CP was more than 8 times faster and had a rolling coverage much closer to the prescribed level. Fifteen other random sequences besides the one presented here were generated from the random walk on the cycle and the conclusion was invariably the same.

\begin{figure}[b]
    \centering
    \includegraphics[width=\textwidth]{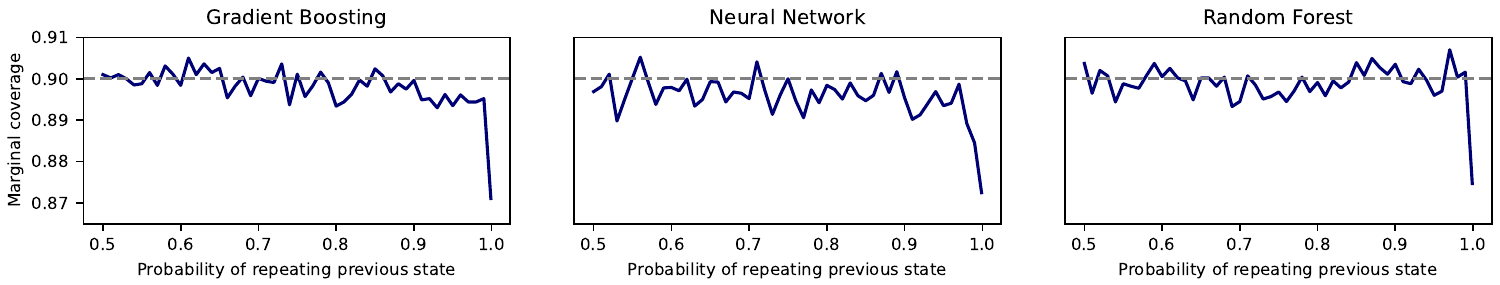}
    \caption{Marginal coverage for two-state hidden Markov model (solid) and nominally prescribed $90\%$ target (dashed) for different levels of dependence and three different models. Coverage is above $89\%$ for all but very extreme levels.}
    \label{fig:cp_test_point-two_state_markov_chain}
\end{figure}

\textit{Example 3 (Two-state hidden Markov model)}
Let $(W_0, W_1, W_2, \ldots)$ be a Markov chain with state space $\sW = \{0,1\}$, probabilities $\P[W_t=1|W_{t-1}=0]=p$ and $\P[W_t=0|W_{t-1}=1]=q$, following the stationary distribution with $\pi = \begin{bmatrix} \tfrac{q}{p+q} & \tfrac{p}{p+q} \end{bmatrix}$, with $p,q \in (0,1)$ and $p+q>0$. This data is stationary $\beta$-mixing, and the mixing coefficients can be found explicitly \citep{mcdonald2015estimating}. When $p = q = 0.5$, $\beta(r) \equiv 0$ for all $r \in \N_{>0}$, so the Markov chain reduces to a sequence of iid Bernoulli trials. On the other hand, as $p$ and $q$ tend towards zero, $\beta(r)$ becomes large for every $r$ and dependence increases. We construct a hidden Markov model by adding a Gaussian noise with zero mean and variance of $10^{-6}$ to the Markov chain with $p=q$, and consider predicting a single data point by using the 11 preceding elements in the series as features. Figure~\ref{fig:cp_test_point-two_state_markov_chain} shows how marginal coverage (\ref{eq:marginal_coverage_guarantee}), calculated through $10\,000$ simulations, is affected by increasing levels of dependence $1-p$ for three different models (boosting, neural network and random forest) and $n_{\train}=1000, n_{\cal}=500$ and $n_{\test}=1$. We note in passing that the $\text{AR}(1)$ experiment (cf. Figure~\ref{fig:intro}) used the same amount of data for training, calibration and testing. Marginal coverage observed is close to nominal iid value of 90\% for the independent case ($p = 0.5$) and weak to medium dependence, measured by the probabilities $1-p$ of repeating the previous state. Coverage remains above 89\% even for large values of dependence, and falls below 88\% only after $1-p = 0.999$. Also, Figure~\ref{fig:eta} shows how the correction $\varepsilon_{\cal}+\delta_{\cal}+\varepsilon_{\train}$ in Theorem \ref{theorem:marginal_coverage_test_data} depends on the calibration set sizes, quickly converging to the iid limit even for moderately dependent data.

\begin{figure}[t]
        \centering
        \includegraphics[width=.7\textwidth]{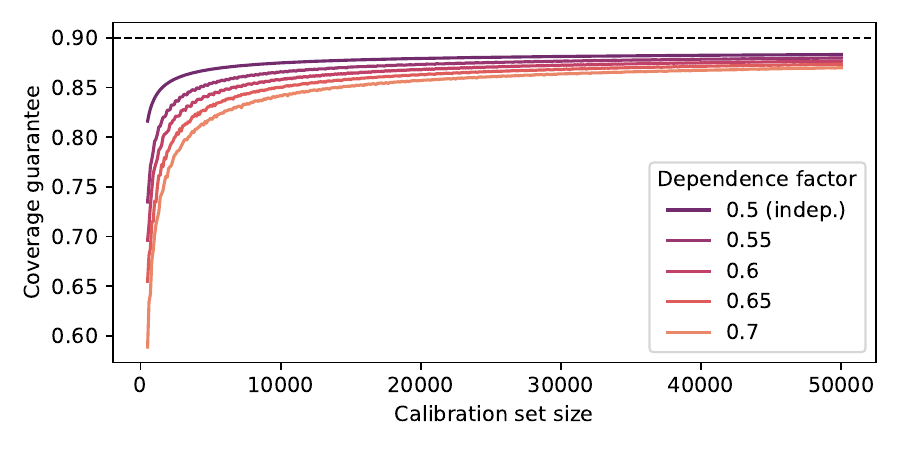}
        \caption{Marginal coverage for different calibration set sizes and dependence levels; the guarantees under dependence converge to the iid case with larger calibration sizes.}
        \label{fig:eta}
\end{figure}

\begin{figure}[hb]
        \centering
        \includegraphics[width=.7\textwidth]{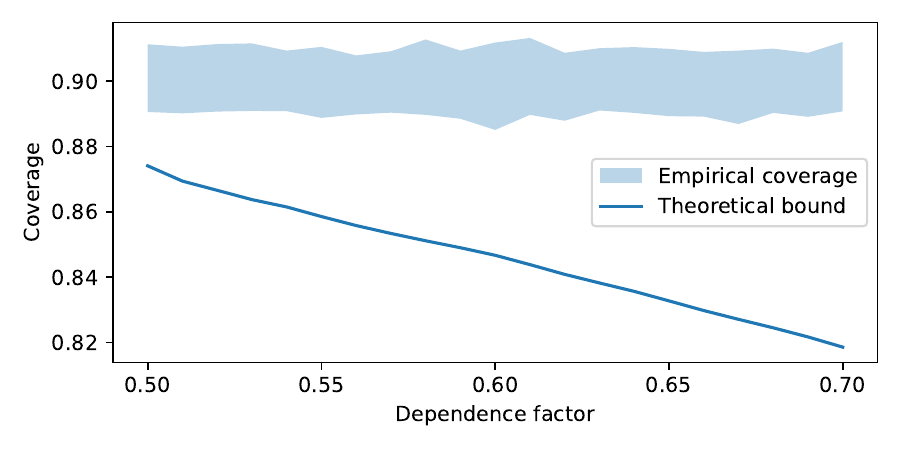}
        \caption{Empirical coverage bound; while the worst-case bound decreases with dependence, empirical coverage remains close to the iid level.}
        \label{fig:cp_test_set-alpha_0.1-n1_1000-n2_15000-n3_15000-lags_10}
\end{figure}

Further, Figure~\ref{fig:cp_test_set-alpha_0.1-n1_1000-n2_15000-n3_15000-lags_10} shows the empirical coverage for a gradient boosting model over a thousand simulations with $n_{\train}=1000$, $n_{\cal}=n_{\test}=15\,000$ and $\delta_{\cal}=\delta_{\test}=0.005$. Note the empirical coverage revolves around the prescribed iid level $0.9$, and it remains above the worst-case theoretical bound, which decreases with the dependence level $1-p$.

\textit{Example 4 (Financial time series)}
We study split CP's performance on three real-world time series: the euro spot exchange rate (\texttt{eurusd}), Brent crude oil future (\texttt{bcousd}) and S\&P 500 stock index future (\texttt{spxusd}). Minute-by-minute data is retrieved directly from \citet{histdata} via an open-source API \citep{histdata-api}. We compute linear returns by dividing a price at minute $t$ by the price at minute $t-1$ and subtracting $1$. Due to market closures, Fridays and Sundays were discarded. We use gradient boosting to predict the price at time $t$ using the prices at times $t-11,\ldots,t-1$. Then, we apply online conformal prediction over a sliding window of $1000$ training points, $500$ calibration points and $1$ single test point for the entire year of 2021. Figure~\ref{fig:marginal_coverage-financial_data} shows the daily coverage of split CP. The dashed black line represents the iid nominal coverage of 90\% and the dashed orange one the marginal coverage over the entire year. Marginal coverage is slightly below 90\%, but never drastically so.

\begin{figure}[t]
    \centering
    \includegraphics[width=\textwidth]{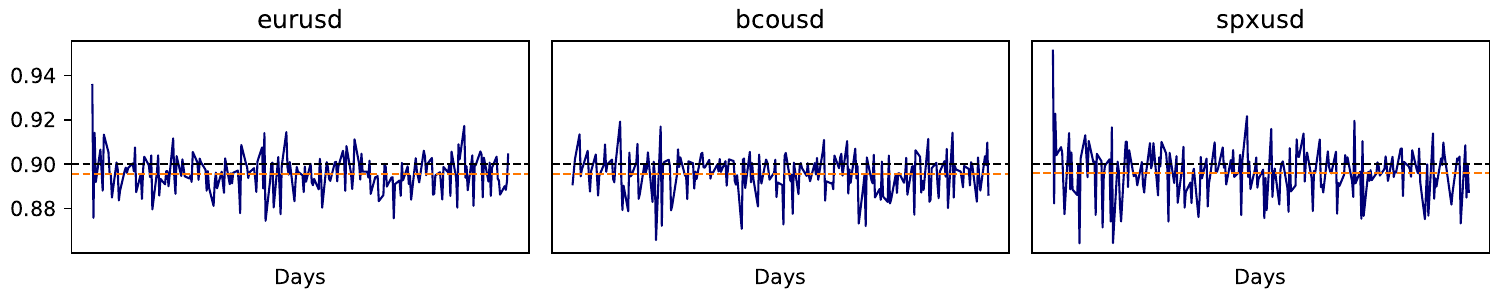}
    \caption{Daily marginal coverages of minute-by-minute online prediction for financial time series \texttt{eurusd}, \texttt{bcousd} and \texttt{spxusd} (solid blue), prescribed iid levels of $1 - \alpha = 0.9$ (dashed black) and observed marginal coverages over the entire year (dashed orange), above $0.895$ in all cases.}
    \label{fig:marginal_coverage-financial_data}
\end{figure}

\begin{table}[htb]
\footnotesize
\centering
\begin{tabular}{P{1.8cm}P{1.8cm}P{1.8cm}P{1.8cm}P{1.8cm}P{1.8cm}}
\toprule
\multirow{2}{*}{Data set}          & \multirow{2}{*}{Cal. set size} & \multicolumn{4}{c}{Conditional coverage} \\ \cmidrule(lr){3-6}
                                  &                                & Uptrend & Downtrend & High vol. & Low vol.         \\ \cmidrule(lr){1-6}
\multirow{3}{*}{\texttt{eurusd}}  & \phantom{0}500                 & 88.76\% &   88.82\% &   87.64\% &  90.07\%         \\
                                  & 1000                           & 89.19\% &   89.17\% &   88.38\% &  90.19\%         \\
                                  & 5000                           & 90.03\% &   89.98\% &   89.85\% &  90.08\%         \\ \cmidrule(lr){1-6}
\multirow{3}{*}{\texttt{bcousd}}  & \phantom{0}500                 & 88.94\% &   88.72\% &   87.10\% &  89.43\%         \\
                                  & 1000                           & 89.35\% &   89.04\% &   87.65\% &  89.95\%         \\
                                  & 5000                           & 89.78\% &   89.77\% &   89.33\% &  89.98\%         \\ \cmidrule(lr){1-6}
\multirow{3}{*}{\texttt{spxusd}}  & \phantom{0}500                 & 89.12\% &   89.01\% &   88.87\% &  89.68\%         \\
                                  & 1000                           & 89.53\% &   89.48\% &   88.84\% &  90.03\%         \\
                                  & 5000                           & 90.04\% &   89.73\% &   89.53\% &  90.30\%         \\
\bottomrule
\end{tabular}
\caption{Conditional coverage for distinct trend and volatility events and varying calibration set size (before conditioning). Note that conditional coverage is generally close to nominal iid level $1-\alpha=0.9$ and results improve given more calibration points.\label{tab:conditional}}
\end{table}

Table~\ref{tab:conditional} presents the conditional coverage (\ref{eq:conditional_coverage_guarantee}) on four events of interest for all three financial data sets. Uptrend (respectively, downtrend) stands for two consecutive observations of positive (negative) returns. High (low) volatility events are taken to be those in which the standard deviation of the previous $10$ returns observed is above (below) a given threshold. Note that conditioning on all such events still yields coverage close to the nominal iid level, on all three data sets. As previously noted, larger calibration sets have an important effect in improving coverage.

All experiments in this section were also conducted for nominally
prescribed iid levels of 85\% and 95\% and the conclusions remained the same, as expected.

\section{Conclusion}

This paper shows that many of the appealing properties of split conformal prediction hold well even when the data is not exchangeable. Theorems \ref{theorem:marginal_coverage_test_data} and \ref{theorem:empirical_coverage_test_data} give marginal and empirical coverage guarantees for broad classes of non-exchangeable data, and Theorem \ref{theorem:marginal_coverage_test_data-conditional} gives marginal coverage guarantees. We also consider the concrete case of stationary $\beta$-mixing sequences in Theorems \ref{theorem:general_framework_1_beta_mixing} and \ref{theorem:general_framework_2_beta_mixing}, and show that our error bounds for such non-exchangeable data may match the order of the iid bounds. To prove these results, we introduce a novel mathematical framework that frees split CP from its exchangeability hypotheses, replacing it with concentration inequalities and decoupling properties.

We further show that these guarantees translate to robust empirical performance by split CP beyond independent data, with examples included in Section \ref{sec:empirical_studies} ranging from synthetic to real data. In the same section, in spite of its generality, split CP is shown to fare well when compared to recent conformal algorithms developed for specific kinds of non-exchangeability. Except for extreme levels of dependence or very small calibration sets---small ``effective sample size'' regimes---, performance holds exceedingly well for standard split CP without any modifications.

Finally, we note that the general framework introduced in Section~\ref{sec:general_framework} can be extended in many directions. Section \ref{sec:extensions} looks at how results generalize for non-stationary data, as well as for other conformal prediction methods that are not based on split CP. One particular open direction is the application of concentration inequalities for weighted sums of exchangeable random variables (e.g., \citealt{barber2024hoeffding}) to analyze weighted conformal methods such as NexCP \citep{barber2023conformal}. More generally, we expect that several results and methods (e.g., \citealt{chernozhukov2021distributional,chernozhukov2018blocks,zaffran2022adaptive,feldman2022risk}) can be analyzed and extended via our unified framework, which remains an avenue for future work.

\acks{RIO is supported by CNPq grants 432310/2018-5 (Universal) and 304475/2019-0 (Produtividade em Pesquisa), and FAPERJ grants 202.668/2019 (Cientista do Nosso Estado)  and  290.024/2021 (Edital Intelig\^{e}ncia Artificial). PO is supported by FAPERJ grant SEI-260003/001545/2022.}

\appendix

\section{Proofs of Section \ref{sec:general_framework}}

For the proofs below, we need to introduce certain population quantiles for $\widehat{s}_{\train}(X_*,Y_*)$ conditionally on the training data.

\begin{definition}[Conditional $\phi$-quantile of the conformity score]\label{definition:marginal_quantile}
Given $\phi\in [0,1)$, let $q_{\phi,\train}$ denote the $\phi$-quantile of $\widehat{s}_{\train}(X_*,Y_*)$ conditioned on the training data; that is:
\[q_{\phi,\train} \coloneqq \inf\{t\in\R\,:\,\P[\widehat{s}_{\train}(X_*,Y_*)\leq t\mid \{(X_i,Y_i)\}_{i\in I_{\train}}]\geq \phi\}.\]
Alternatively, define, for a deterministic $(x_i,y_i)_{i=1}^{n_{\train}}\in(\sX\times \sY)^{n_{\train}}$, the $\phi$-quantile:
\[q_{\phi}((x_i,y_i)_{i=1}^{n_{\train}}) \coloneqq \inf\{t\in\R\,:\,\P[s((x_i,y_i)_{i=1}^{n_{\train}},(X_*,Y_*))\leq t]\geq \phi\},\]
and set $q_{\phi,\train} \coloneqq q_{\phi}((X_i,Y_i)_{i\in I_{\train}})$. We also define:
\[p_{\phi,\train} \coloneqq \P[\widehat{s}_{\train}(X_*,Y_*)\leq t\mid (X_i,Y_i)_{i\in I_{\train}}]\]
\end{definition}

\begin{remark}When the conditional law of $\widehat{s}_{\train}(X_*,Y_*)$ given the training data is continuous, we have $p_{\phi,\train}=\phi$. Otherwise, it only holds that $p_{\phi,\train}\geq \phi$.\end{remark}

\begin{proof}[Theorem \ref{theorem:marginal_coverage_test_data}]
First we show that the event $F=\{q_{1-\alpha-\varepsilon_{\cal},\train}\leq \widehat{q}_{1-\alpha,\cal}\}$ satisfies $\P[F]\geq 1-\delta_{\cal}.$
Indeed, by (\ref{eq:quantile}) and Definition~\ref{definition:marginal_quantile}, if condition (\ref{eq:hypothesis-concentration_over_calibration}) holds, for any $\ell\in \N_{>0}$, with probability at least $1-\delta_{\cal}$,
  \begin{align*}
      \frac{1}{n_{\cal}}\sum_{i\in I_{\cal}} \Ind{\widehat{s}_{\train}(X_i,Y_i) \leq q_{1-\alpha-\varepsilon_{\cal},\train}-1/\ell}&\leq \P[\widehat{s}_{\train}(X_*,Y_*)\leq q_{1-\alpha-\varepsilon_{\cal},\train}-1/\ell]+\varepsilon_{\cal}\\
      &<  1-\alpha-\varepsilon_{\cal}+\varepsilon_{\cal}=1-\alpha\\
      &\leq \frac{1}{n_{\cal}}\sum_{i\in I_{\cal}} \Ind{\widehat{s}_{\train}(X_i,Y_i) \leq \widehat{q}_{1-\alpha,\cal}}.
  \end{align*}
 This implies that  the event
\[E_{\ell}=\left\{       \frac{1}{n_{\cal}}\sum_{i\in I_{\cal}} \Ind{\widehat{s}_{\train}(X_i,Y_i) \leq q_{1-\alpha-\varepsilon_{\cal},\train}-1/\ell}\\
< \frac{1}{n_{\cal}}\sum_{i\in I_{\cal}} \Ind{\widehat{s}_{\train}(X_i,Y_i) \leq \widehat{q}_{1-\alpha,\cal}}\right\}\]
satisfies $\P[E_\ell]\geq 1-\delta_{\cal}$ for all $\ell\in\N_{>0}$, and since $E_{\ell+1}\subset E_\ell$, we have
$1-\delta_{\cal}\leq \lim_{\ell\to\infty} \P[E_\ell] = \P[E_{\infty}]$,
where,
\[E_\infty=\left\{       \frac{1}{n_{\cal}}\sum_{i\in I_{\cal}} \Ind{\widehat{s}_{\train}(X_i,Y_i) \leq q_{1-\alpha-\varepsilon_{\cal},\train}}\\
\leq \frac{1}{n_{\cal}}\sum_{i\in I_{\cal}} \Ind{\widehat{s}_{\train}(X_i,Y_i) \leq \widehat{q}_{1-\alpha,\cal}}\right\},\]
proving that  $\P[F]\geq 1-\delta_{\cal}$.
Therefore, given  $i \in I_{\test}$, using the fact that the function $t\mapsto\P[\widehat{s}_{\train}(X_i, Y_i)\leq t]$ is increasing,
\begin{align*}
    \P[\widehat{s}_{\train}(X_i,Y_i)\leq \widehat{q}_{1-\alpha,\cal}] &\geq  \P[\{\widehat{s}_{\train}(X_i,Y_i)\leq \widehat{q}_{1-\alpha,\cal}\}\cap F]\\
    &\geq \P[\widehat{s}_{\train}(X_i,Y_i)\leq q_{1-\alpha-\varepsilon_{\cal},\train}] - \delta_{\cal}.
\end{align*}
Hence, by condition (\ref{eq:hypothesis-decoupling_over_test_data}) and a conditioning argument,
\begin{align}\label{eq:conclusion_first_part_marginal_coverage}
    \P[\widehat{s}_{\train}(X_i,Y_i)\leq \widehat{q}_{1-\alpha,\cal}] &\geq \P[\widehat{s}_{\train}(X_i,Y_i)\leq q_{1-\alpha-\varepsilon_{\cal},\train}] - \delta_{\cal}\\
    &\geq \P[\widehat{s}_{\train}(X_*,Y_*)\leq q_{1-\alpha-\varepsilon_{\cal},\train}] - \varepsilon_{\test} - \delta_{\cal}\nonumber\\
    & \geq 1-\alpha-\varepsilon_{\cal}- \varepsilon_{\test} - \delta_{\cal}\nonumber,
\end{align}
proving the  first part of the theorem.

For the second part, note that  by  Definition  \ref{definition:marginal_quantile} and condition (\ref{eq:hypothesis-concentration_over_calibration}) we have with probability at least $1-\delta_{\cal}$,
  \begin{equation*}
      \frac{1}{n_{\cal}}\sum_{i\in I_{\cal}}\Ind{\widehat{s}_{\train}(X_i,Y_i)\leq q_{1-\alpha +\varepsilon_{\cal},\train }} \geq  \P[\widehat{s}_{\train}(X_*,Y_*)\leq q_{1-\alpha+\varepsilon_{\cal},\train}]-\varepsilon_{\cal}
      \geq  1-\alpha,
  \end{equation*}
  and since $\widehat{q}_{1-\alpha-\varepsilon_{\cal},\cal}$ is the smallest possible value satisfying the expression above,   the event  $G=\{\widehat{q}_{1-\alpha,\cal}\leq q_{1-\alpha+\varepsilon_{\cal},\train}\}$ satisfies $\P[G]\geq 1-\delta_{\cal}.$ Then,
  \begin{align*}
      \P[\widehat{s}_{\train}(X_i,Y_i)\leq \widehat{q}_{1-\alpha,\cal}] &\leq        \P[\{\widehat{s}_{\train}(X_i,Y_i)\leq \widehat{q}_{1-\alpha,\cal}\}\cap G] +\delta_{\cal}\\
      &\leq \P[\widehat{s}_{\train}(X_i,Y_i)\leq q_{1-\alpha+\varepsilon_{\cal},\train}] +\delta_{\cal}.
  \end{align*}
Hence, if $\widehat{s}_{\train}(X_*,Y_*)$ almost surely has a continuous distribution conditionally on the training data, by condition (\ref{eq:hypothesis-decoupling_over_test_data})
  \begin{align*}
      \P[\widehat{s}_{\train}(X_i,Y_i)\leq \widehat{q}_{1-\alpha,\cal}]  &\leq \P[\widehat{s}_{\train}(X_*,Y_*)\leq q_{1-\alpha+\varepsilon_{\cal},\train}] +\varepsilon_{\test}+\delta_{\cal}\\
      & = 1-\alpha + \varepsilon_{\cal}+\varepsilon_{\test}+\delta_{\cal}.
  \end{align*}
Putting this together with (\ref{eq:conclusion_first_part_marginal_coverage}) concludes the second part.
\end{proof}

\begin{proof}[Theorem \ref{theorem:empirical_coverage_test_data}]
Assuming that condition (\ref{eq:hypothesis-concentration_over_calibration}) and condition (\ref{eq:hypothesis-concentration_of_empirical_cdf}) hold and using a similar argument  as we did in the proof of Theorem \ref{theorem:marginal_coverage_test_data}, it is straightforward to show that the event $F=\{\widehat{q}_{1-\alpha-\varepsilon_{\cal}-\varepsilon_{\test},\test}\leq \widehat{q}_{1-\alpha,\cal}\}$
satisfies $\P[F]\geq 1-\delta_{\cal}-\delta_{\test}.$ But then,
\begin{align*}
    &\P\left[\frac{1}{n_{\test}} \sum_{i \in I_{\test}} \Ind{\widehat{s}_{\train}(X_i, Y_i) \leq \widehat{q}_{1-\alpha, \cal}} \geq 1 - \alpha - \varepsilon_{\cal}-\varepsilon_{\test}\right]\\
    &\geq \P\left[\frac{1}{n_{\test}} \sum_{i \in I_{\test}} \Ind{\widehat{s}_{\train}(X_i, Y_i) \leq \widehat{q}_{1-\alpha-\varepsilon_{\cal}-\varepsilon_{\test},\test}} \geq 1 - \alpha - \varepsilon_{\cal}-\varepsilon_{\test}\right] - \P[F^c]\\
    & \geq 1 - \delta_{\cal}-\delta_{\test},
\end{align*}
proving the first part.
For the second part, note that
$G=\{\widehat{q}_{1-\alpha+\varepsilon_{\cal}+\varepsilon_{\test},\test}\geq \widehat{q}_{1-\alpha,\cal}\}$
also has probability at least $1-\delta_{\cal}-\delta_{\test}$, therefore the event
$F\cap G = \{\widehat{q}_{1-\alpha+\varepsilon_{\cal}+\varepsilon_{\test},\test}\geq \widehat{q}_{1-\alpha-\varepsilon_{\cal}-\varepsilon_{\test},\test}\}$
has probability at least $1-2\delta_{\cal}-2\delta_{\test}$
Hence, if $\widehat{s}_{\train}(X_*,Y_*)$ almost surely has a continuous distribution conditionally on the training data, using the same argument  as we did in the proof of Theorem \ref{theorem:marginal_coverage_test_data} concludes the theorem.
\end{proof}

\par\noindent{\bf Proof\ }[Application to the iid case]
First, note that, in the iid case, when $i\in I_{\test}$,
\begin{equation*}
    \P[\widehat{s}_{\train}(X_i,Y_i)\leq q_{\train}] = \P[\widehat{s}_{\train}(X_*,Y_*)\leq q_{\train}],
\end{equation*}
showing that condition (\ref{eq:hypothesis-decoupling_over_test_data}) holds with $\varepsilon_{\test}=0$.

Moreover, using the fact that $(\Ind{\widehat{s}_{\train}(X_i,Y_i)\leq q_{\train}})_{i=1}^n$ is an iid sample of bounded random variables, by Hoeffding's inequality, with probability at least $1-\delta,$
\[\left|\frac{1}{n} \sum_{i=1}^n \Ind{\widehat{s}_{\train}(X_i,Y_i)\leq q_{\train}}- P_{q,\train} \right|\leq \sqrt{\frac{1}{2n}\log\left(\frac{2}{\delta}\right)}.\]
Therefore, conditions (\ref{eq:hypothesis-concentration_over_calibration}) and (\ref{eq:hypothesis-concentration_of_empirical_cdf}) are proved by taking
\begin{equation*}
    \varepsilon_{\cal}=\sqrt{\frac{1}{2n_{\cal}}\log\left(\frac{2}{\delta_{\cal}}\right)}\quad \textrm{and}\quad
    \varepsilon_{\test} =\sqrt{\frac{1}{2n_{\test}}\log\left(\frac{2}{\delta_{\test}}\right)}.\tag*{\BlackBox}
\end{equation*}

\section{Proofs of Section \ref{sec:conditional_guarantees} and Further Results}
\label{appendix:conditional_guarantees}

Let $\sA$ be a family of measurable subsets of $\sX$. Given $\phi\in [0,1)$ and $A\in\sA$, let $I_{\cal}(A)\coloneqq\{i \in I_{\cal} : X_i \in A\}$ and $n_{\cal}(A)\coloneqq \#I_{\cal}(A)$. Denote the empirical $\phi$-quantile of $\widehat{s}_{\train}(X_i,Y_i)$ over $i\in I_{\cal}$ as:
\begin{equation*}
  \widehat{q}_{\phi,\cal}(A) \coloneqq \inf\left\{t\in\R\,:\, \frac{1}{n_{\cal}(A)}\sum_{i\in I_{\cal}(A)}\Ind{\widehat{s}_{\train}(X_i,Y_i)\leq t}\geq \phi\right\},
\end{equation*}
and, for $x\in A$, define the prediction set:
\begin{equation*}
   C_{\phi}(x;A) \coloneqq \{y\in\sY\,:\, \widehat{s}_{\train}(x,y)\leq \widehat{q}_{\phi, \cal}(A)\}.
\end{equation*}
For $A \in \sA $ with $\P[X\in A]>0$, let
\begin{equation} \label{eq:p_q_train-conditional}
    {P}_{q,\train}(A) \coloneqq \P[\widehat{s}_{\train}(X_*,Y_*)\leq q_{\train}\mid (X_i,Y_i)_{i\in I_{\train}},X_*\in A].
\end{equation}

\par\noindent{\bf Proof\ }[Theorem \ref{theorem:marginal_coverage_test_data-conditional}]
Following the same strategy as in Theorem \ref{theorem:marginal_coverage_test_data}, we have that, with probability at least $1-\delta_{\cal}$, the event $F_{\cal} = \{q_{1-\alpha-\varepsilon_{\cal}}(A)\leq \widehat{q}_{1-\alpha,\cal}(A),\ \forall A\in \sA\}$ satisfies $\P[F_{\cal}]\geq 1-\delta_{\cal}.$
Now, using the fact that the function
$t \mapsto \P[\widehat{s}_{\train}(X_k, Y_k) \leq t\mid X_k\in A]$ is increasing, for any $k \in I_{\test}$ and  all $A\in \sA$,
\begin{align*}
    \P[\widehat{s}_{\train}(X_k,Y_k)\leq \widehat{q}_{1-\alpha,\cal}(A)\mid X_k\in A] &\geq  \P[\{\widehat{s}_{\train}(X_k,Y_k)\leq \widehat{q}_{1-\alpha,\cal}(A)\}\cap F_{\cal}\mid X_k\in A]\\
    & \geq \P[\{\widehat{s}_{\train}(X_k,Y_k)\leq q_{1-\alpha-\varepsilon_{\cal}}(A)\}\cap F_{\cal}\mid X_k\in A]\\
    &\geq \P[\widehat{s}_{\train}(X_k,Y_k)\leq q_{1-\alpha-\varepsilon_{\cal}}(A)\mid X_k\in A] - \delta_{\cal}.
\end{align*}
Then, to conclude,
\begin{align*}
    &\P[\widehat{s}_{\train}(X_k,Y_k)\leq \widehat{q}_{1-\alpha,\cal}(A)\mid X_k\in A]\\
    &\geq \P[\widehat{s}_{\train}(X_k,Y_k)\leq q_{1-\alpha-\varepsilon_{\cal}}(A)\mid X_k\in A] - \delta_{\cal}\\
    &\geq \P[\widehat{s}_{\train}(X_*,Y_*)\leq q_{1-\alpha-\varepsilon_{\cal}}(A)\mid X_*\in A] - \varepsilon_{\train} - \delta_{\cal}\\
    &= \E[\P[\widehat{s}_{\train}(X_*,Y_*)\leq q_{1-\alpha-\varepsilon_{\cal}}(A)\mid (X_i,Y_i)_{i\in I_{\train}}, X_*\in A]] - \varepsilon_{\train} - \delta_{\cal}\\
    & \geq 1-\alpha-\varepsilon_{\cal}- \varepsilon_{\train} - \delta_{\cal}.\tag*{\BlackBox}
\end{align*}

We now proceed to give an empirical conditional coverage guarantee. For a conditional version of (\ref{eq:hypothesis-concentration_of_empirical_cdf}), suppose there exist $\delta_{\test}$ and $\varepsilon_{\test} \in(0,1)$ such that
    \begin{equation} \label{eq:hypothesis-concentration_of_empirical_cdf-conditional}
        \P\left[\sup_{A\in\sA}\left|P_{q,\train}(A)- \frac{1}{n_{\test}(A)} \sum_{i\in I_{\test}(A)}\Ind{\widehat{s}_{\train}(X_i,Y_i)\leq q_{\train}}\right|\leq \varepsilon_{\test}\right] \geq 1-\delta_{\test},
    \end{equation}
where $P_{q,\train}(A)$ is defined as in (\ref{eq:p_q_train-conditional}). This suffices for empirical conditional coverage.

\begin{theorem}[Empirical conditional coverage over test data]\label{theorem:empirical_conditional_coverage_test_data}
Given $\alpha\in(0,1)$, $\delta_{\cal}>0$ and $\delta_{\test}>0$, if (\ref{eq:hypothesis-concentration_over_calibration-conditional}) and (\ref{eq:hypothesis-concentration_of_empirical_cdf-conditional}) hold, then for each $A\in \sA$:
    \begin{align*}
        &\P\left[\inf_{A\in\sA}\frac{1}{n_{\test}(A)} \sum_{i \in I_{\test}(A)} \Ind{Y_i\in C_{1-\alpha}(X_i;A)} \geq 1 - \alpha - \eta\right] \geq 1 - \delta_{\cal}-\delta_{\test},
    \end{align*}
where $\eta = \varepsilon_{\cal}+\varepsilon_{\test}.$
Additionally, if $\widehat{s}_{\train}(X_*,Y_*)$ almost surely has a continuous distribution conditionally on the training data, then:
    \begin{align*}
        &\P\left[\sup_{A\in\sA}\left|\frac{1}{n_{\test}(A)} \sum_{i \in I_{\test}(A)} \Ind{Y_i\in C_{1-\alpha}(X_i;A)} - (1 - \alpha)\right| \leq  \eta\right] \geq 1 - 2\delta_{\cal}-2\delta_{\test}.
    \end{align*}
\end{theorem}
\begin{proof}
As in the proof of empirical coverage over test data (Theorem \ref{theorem:empirical_coverage_test_data}), the event $F=\{\widehat{q}_{1-\alpha-\varepsilon_{\cal}-\varepsilon_{\test},\test}(A)\leq \widehat{q}_{1-\alpha,\cal}(A) \forall A\in \sA\}$
satisfies $\P[F]\geq 1-\delta_{\cal}-\delta_{\test}$ and the remaining of the proof is just a direct application of the definition of conditional empirical quantile calibrated over the test data.
\end{proof}

The results above are directly applicable in the iid case. It is straightforward to show that (\ref{eq:hypothesis-decoupling_over_test_data-conditional}) holds with $\varepsilon_{\test}=0$ and, if the family $\sA$ has finite VC dimension $\text{VC}(\sA)=d$ and $\P[A]>\gamma$ for some $\gamma>0$ and all $A\in\sA$, it suffices to take $\varepsilon_{\cal}=\gamma^{-1}(4 \sqrt{\log(2(n+1)^{d})/n}+2 \sqrt{\log(4/\delta)/(2n)})$.

\section{Proofs of Section \ref{sec:stationary_beta_mixing_data} and Further Results}
\label{appendix:stationary_beta_mixing_data}

\subsection{Standard Coverage Guarantees}\label{appendix:beta_mixing-standard_coverage_guarantees}

Our goal is to check conditions (\ref{eq:hypothesis-concentration_over_calibration}), (\ref{eq:hypothesis-decoupling_over_test_data}) and (\ref{eq:hypothesis-concentration_of_empirical_cdf}) when $(X_i,Y_i)_{i=1}^n$ is a stationary $\beta$-mixing process. As stated in the main text, the main tool we will use is the so-called Blocking Technique \citep{yu1994rates, mohri2010stability, kuznetsov2017generalization}. It allows one to measure the difference in expectation between a function of a $\beta$-mixing process and the same function over an independent process, thereby transforming the original dependent problem into an independent one with the addition of a penalty factor.

\begin{proposition}[Blocking Technique] \label{prop:blocking_technique}
    Let $\{Z_t\}_{t=1}^{T}$ be a sample of a stationary $\beta$-mixing process. Split the sample into $2m$ interleaved blocks, with even blocks of size $a$ and odd blocks of size $b$, such that $T=m(a+b)$. Denote each block by $B_j = \{Z_i\}_{i=l(j)}^{u(j)}$, where $l(j)=1+ \lceil (j-2)/2 \rceil a+\lfloor j/2 \rfloor b$ and $u(j)=\lfloor j/2 \rfloor a + \lceil j/2 \rceil b $, so the set of odd blocks, each of size $b$, is given by $B_{\mathrm{odd}}=(B_1, B_3, \ldots, B_{2m-1})$. Consider also the set $B^{*}_{\mathrm{odd}}=(B_1^{*}, B_3^{*}, \ldots, B_{2m-1}^{*})$ where $B^{*}_{j}$ are independent for $j=1, 3, \ldots, 2m-1$, and $B^{*}_j \stackrel{d}{=} B_j$. If $h:\R^{mb}\to\R$ is a Borel-measurable function with $|h|\leq M$ for some $M>0$, then
    \begin{equation*}
        |\E[h(B_{\mathrm{odd}})] - \E[h(B^{*}_{\mathrm{odd}})]| \leq 2M(m-1)\beta(a),
    \end{equation*}
    where $\beta(a)$ is the $\beta$-mixing coefficient of $\{Z_t\}_{t=1}^{T}$.
\end{proposition}

Using the Blocking Technique, we can prove that up to a error correction factor, we can transform our stationary $\beta$-mixing problem into an iid one:

\begin{lemma}[\citealt{mohri2009rademacher}]\label{lemma:uniform_bound_beta_mixing}
    Let $Z_1, \ldots, Z_n$ be a sample drawn from a stationary $\beta$-mixing distribution and $\sF$ be a class of functions from $\sX$ to $[0,1]$. Split the sample into $2m$ blocks, with blocks of size $a$ with $n=2ma$. Denote the blocks by $B_j = \{Z_i\}_{i=l(j)}^{u(j)}$ where $l(j)=1+(j-1)a$ and $u(j)=ja$, with $B_{\mathrm{odd}}=(B_1, B_3, \ldots, B_{2m-1})$. Call the independent version of $B_{\mathrm{odd}}$ by $B_{\mathrm{odd}}^{*}=(B_1^{*}, B_3^{*}, \ldots, B_{2m-1}^{*})$, where $B_j^{*}$ are independent with $B_j^{*}\stackrel{d}{=} B_j$, and let $\P_*$ be their law. Then,
    \begin{align*}
        \P\left[\sup_{f\in\sF}\left|\E[f(Z_1)] - \frac{1}{n}\sum_{i=1}^{n} f(Z_i)\right|>\varepsilon\right]& \leq 2 \P_*\left[\sup_{f\in\sF}\left|\E\left[f(Z_1)\right] - \frac{1}{ma}\sum_{Z_j \in B^{*}_{\mathrm{odd}}} f(Z_j) \right|>\varepsilon\right]\\
        &+ 4(m-1)\beta(a).
    \end{align*}
\end{lemma}

Finally, using Lemma \ref{lemma:uniform_bound_beta_mixing} and Bernstein's Inequality, we are ready to prove a concentration inequality for stationary $\beta$-mixing processes.

\begin{lemma}\label{lemma:bound_beta_mixing-bernstein}
    Let $Z_1, \ldots, Z_n$ be a sample drawn from a stationary $\beta$-mixing distribution with $Z_1 \in [0, 1]$ and $\V[Z_1]=v< \infty$. Then, for any $m, a,s\in \N_{+}$ with $m>1$, $n=2ma+s$ and $\delta>4(m-1)\beta(a))$, with probability at least $1-\delta$ it holds that
    \begin{equation*}
        \left|\E\left[Z_1\right] - \frac{1}{n}\sum_{i=1}^{n} Z_i\right| \leq \varepsilon ,
    \end{equation*}
    where
    \[\varepsilon= \widetilde{\sigma}(a)\sqrt{\frac{4}{n} \log\left(\frac{4}{\delta-4(m-1)\beta(a)}\right)} + \frac{1}{3m}\log\left(\frac{4}{\delta-4(m-1)\beta(a)}\right)+ \frac{s}{n},\]
    and
    $\widetilde{\sigma}(a)=\sqrt{v +\frac{2}{a}\sum_{k=1}^{a-1}(a-k)\beta(k)}.$
\end{lemma}
\begin{proof}
 By an application of Lemma \ref{lemma:uniform_bound_beta_mixing} taking $\sF = \{x\mapsto x, x\in [0,1]\}$
        and Bernstein's inequality over the $m$ independent blocks,  with probability at least $1-\delta$
    \begin{align}\label{appendix:bound_bernstein}
        \left|\E\left[Z_1\right] - \frac{1}{n}\sum_{i=1}^{n} Z_i\right| & \leq \frac{2ma}{n}\left|\E\left[Z_1\right] - \frac{1}{2ma}\sum_{i=1}^{2ma} Z_i\right| + \frac{s}{n}\\
        &\leq \sigma\sqrt{\frac{2}{m} \log\left(\frac{4}{\delta-4(m-1)\beta(a)}\right)}+\frac{1}{3m}\log\left(\frac{4}{\delta-4(m-1)\beta(a)}\right) + \frac{s}{n},\nonumber
    \end{align}
    where $\sigma^2 = \V\left[\frac{1}{a}\sum_{i: Z_i \in B_j} Z_i\right].$
To estimate $\sigma^2$, note that by stationarity,
\begin{align*}
    \V\left[\frac{1}{a}\sum_{i=1}^a Z_i\right] & = \frac{1}{a}\E\left[Z_1^2 \right] - \E[Z_1]^2+\frac{1}{a^2}\sum_{k=1}^{a-1}\sum_{|i- j|=k}^a \E\left[Z_i Z_j\right].
\end{align*}
Now, using the fact that $\{Z_i\}_{i\in B_j}$ is $\beta$-mixing, we have
\begin{align*}
    \V\left[\frac{1}{a}\sum_{i=1}^a Z_i\right]
    &\leq \frac{1}{a}\E\left[Z_1^2 \right] - \E[Z_1]^2+\frac{1}{a^2}\sum_{k=1}^{a-1}\sum_{|i- j|=k}^a \left(\E\left[Z\right]^2 +\beta(k)\right)\\
    &= \frac{1}{a}\E\left[Z_1^2 \right] - \E[Z_1]^2+\frac{1}{a^2}\sum_{k=1}^{a-1}\sum_{|i- j|=k}^a \E\left[Z\right]^2 +\frac{1}{a^2}\sum_{k=1}^{a-1}\sum_{|i- j|=k}^a \beta(k)\\
  &=\frac{1}{a}\V[Z_1] +\frac{1}{a^2}\sum_{k=1}^{a-1}\sum_{|i- j|=k} \beta(k) = \frac{1}{a}\left(\V[Z_1] +\frac{2}{a}\sum_{k=1}^{a-1}(a-k)\beta(k)\right).
\end{align*}
That is,
\[\sigma \leq \sqrt{\frac{1}{a}\left(v +\frac{2}{a}\sum_{k=1}^{a-1}(a-k)\beta(k)\right)}.\]
Plugging the above expression in (\ref{appendix:bound_bernstein}) and using the fact that $2ma=n,$ yields the result.
\end{proof}

Now we are ready to prove conditions (\ref{eq:hypothesis-concentration_over_calibration}), (\ref{eq:hypothesis-decoupling_over_test_data}) and (\ref{eq:hypothesis-concentration_of_empirical_cdf}) for the stationary $\beta$-mixing case.

\begin{proposition}\label{proposition:assumption_concentration_cal_holds}
If $(X_i,Y_i)_{i=1}^n$ is stationary $\beta$-mixing, then  condition (\ref{eq:hypothesis-concentration_over_calibration}) holds  with
\begin{align*}
        \varepsilon_{\cal}&= \inf_{(a,m,r)\in F_{\cal}}\left\{ \widetilde{\sigma}(a)\sqrt{\frac{4}{n_{\cal}-r+1} \log\left(\frac{4}{\delta_{\cal}-4(m-1)\beta(a)-\beta(r)}\right)}\right.\\
    &\left.+ \frac{1}{3m}\log\left(\frac{4}{\delta_{\cal}-4(m-1)\beta(a)-\beta(r)}\right) +  \frac{r-1}{n_{\cal}}\right\}
\end{align*}
for
\[F_{\cal} = \{(a,m,r)\in \N_{>0}^3\ :\ 2ma = n_{\cal}-r+1,\ \delta_{\cal}>4(m-1)\beta(a)+\beta(r)\},\]
where $\widetilde{\sigma}(a)=\sqrt{\frac{1}{4} +\frac{2}{a}\sum_{k=1}^{a-1}(a-k)\beta(k)}.$
\end{proposition}
\begin{proof}
We want to use Lemma \ref{lemma:bound_beta_mixing-bernstein} for the random variables $\left\{\Ind{\widehat{s}_{\train}(X_i,Y_i)\leq q_{\train }}\right\}_{i\in I_{\cal}},$
however, since  the random variables $(X_i,Y_i)_{i\in I_{\cal}}$ are dependent to  $\widehat{s}_{\train}$ and the quantile $q_{\train}$, we cannot simply apply the result. To fix this problem, it will be necessary to create a gap between our training and calibration data and use the Blocking Technique, Proposition \ref{prop:blocking_technique}, to transpose our problem to an independent setting.

For $\varepsilon>0$ and $r\in \{1,\dots,n_{\cal}\}$, let $I_{\cal,r}=\{n_{\train}+r,\dots,n_{\train}+n_{\cal}\}$ and define the event
\[E(r,\varepsilon) =\left\{\left|\frac{1}{n_{\cal}-r+1}\sum_{i\in I_{\cal,r}}\Ind{\widehat{s}_{\train}(X_i,Y_i)\leq q_{\train }} - P_{q,\train}\right|> \varepsilon\right\},\]
we want to show that there exists $\varepsilon>0$ such that $\P[E(1,\varepsilon)]\leq \delta.$
Note that if $E(1,\varepsilon)$ holds, then
\[\left|\frac{1}{n_{\cal}}\sum_{i\in I_{\cal,r}}\Ind{\widehat{s}_{\train}(X_i,Y_i)\leq q_{\train }} - P_{q,\train}\right|> \varepsilon-\frac{r-1}{n_{\cal}},\]
and since $n_{\cal}\geq n_{\cal}-r+1$,
\[\left|\frac{1}{n_{\cal}-r+1}\sum_{i\in I_{\cal,r}}\Ind{\widehat{s}_{\train}(X_i,Y_i)\leq q_{\train }} - P_{q,\train}\right|> \varepsilon-\frac{r-1}{n_{\cal}},\]
that is, $E(1,\varepsilon)\subset E(r,\varepsilon-(r-1)/n_{\cal})$.
Now, define
$\P_* = \P_{1}^{n_{\train}} \otimes \P_{n_{\train}+r}^{n_{\train}+n_{\cal}},$
so under $\P_*$, we have that $(X_i,Y_i)_{i\in I_{\train}}$ and $(X_i,Y_i)_{i\in I_{\cal},r}$ are independent. Then, by Proposition \ref{prop:blocking_technique},
\begin{align*}
    \P[E(1,\varepsilon)]&\leq \P[E(r,\varepsilon-(r-1)/n_{\cal})]\\
    &\leq \P_*[E(r,\varepsilon-(r-1)/n_{\cal})] + \beta(r)\\
    & = \E_{*}[\P_{*}[E(r,\varepsilon-(r-1)/n_{\cal})\mid (X_i,Y_i)_{i\in I_{\train}}]] + \beta(r).
\end{align*}
Note that by Lemma \ref{lemma:bound_beta_mixing-bernstein}, for any $m, a\in \N_{+}$ with  $n_{\cal}-(r+s)+1=2ma$ and $\delta_{\cal}>4(m-1)\beta(a))$, using the fact that  $\V\left[\Ind{\widehat{s}_{\train}(X_i,Y_i)\leq q_{\train }}\right]\leq 1/4$,  taking
\begin{equation*}
    \varepsilon= \widetilde{\sigma}(a)\sqrt{\frac{4}{n_{\cal}-r+1} \log\left(\frac{4}{\delta-4(m-1)\beta(a)}\right)}+ \frac{1}{3m}\log\left(\frac{4}{\delta-4(m-1)\beta(a)}\right) + \frac{r-1}{n_{\cal}},
\end{equation*}
implies that
$\P_{*}[E(r,\varepsilon-(r-1)/n_{\cal})\mid (X_i,Y_i)_{i\in I_{\train}}] \leq \delta_{\cal}.$
Therefore,
$\P[E(1,\varepsilon)] \leq \delta_{\cal}+ \beta(r),$
which is equivalent to
$\P[E(1,\varepsilon')] \leq \delta_{\cal}$
taking
\begin{align*}
    \varepsilon'&= \widetilde{\sigma}(a)\sqrt{\frac{4}{n_{\cal}-r+1} \log\left(\frac{4}{\delta-4(m-1)\beta(a)-\beta(r)}\right)}\\
    &+ \frac{1}{3m}\log\left(\frac{4}{\delta-4(m-1)\beta(a)-\beta(r)}\right) + \frac{s+r-1}{n_{\cal}}.
\end{align*}
Finally, since this is true for any choice of $a,m,r\in\N_{>0}$ and $s\in\N$ with $s+2ma=n_{\cal}-r+1$ and $\delta>4(m-1)\beta(a)+\beta(r)$, we can choose $a,m,r$ optimally such that the value of $\varepsilon'$ is minimized and there is no need to optimize in $s$ in this case.
\end{proof}

\begin{proposition}\label{proposition:assumption_test_to_iid_holds}
If $(X_i,Y_i)_{i=1}^n$ is stationary $\beta$-mixing, then condition (\ref{eq:hypothesis-decoupling_over_test_data}) holds with
$\varepsilon_{\train} = \beta(k-n_{\train}).$
Moreover, since $\beta(k-n_{\train})\leq\beta(1-n_{\train})$, it is possible to find $\varepsilon_{\train}$ not depending on $k.$
\end{proposition}

\begin{proof}
Given $k\in I_{\test}$,  define  $\P_* = \P_{1}^{n_{\train}} \otimes \P_{k}^{k},$
so under $\P_*$ the random variable $(X_k,Y_K)$ is independent of the training data $(X_i,Y_i)_{i\in I_{\train}}.$
Then, if
$\beta_k = \beta(k-n_{\train})$ we have,
    \begin{align*}
        \beta_k &\geq \left|\P[\widehat{s}_{\train}(X_k, Y_k) \leq {q}_{\train}]- \P_{*}[\widehat{s}_{\train}(X_k, Y_k) \leq {q}_{\train}]\right|\\
        &= \left| \P[\widehat{s}_{\train}(X_k, Y_k) \leq {q}_{ \train}] - \E_*[\P_{*}[\widehat{s}_{\train}(X_k, Y_k) \leq {q}_{\train}\mid (X_i,Y_i)_{i\in I_{\train}}]]\right|\\
        &=\left|\P[\widehat{s}_{\train}(X_k, Y_k) \leq {q}_{ \train}]- \P[\widehat{s}_{\train}(X_*, Y_*) \leq {q}_{ \train}\mid (X_i,Y_i)_{i\in I_{\train}}]\right|,
    \end{align*}
    where the $\beta_k$ penalty follows from Proposition \ref{prop:blocking_technique}.
    Note that the larger the $k$ the smaller the penalty incurred by the dependence in the $\beta$-mixing process. Moreover, since $\beta_k \leq \beta_1,$
    it is possible to define $\varepsilon_{\train}=\beta_1$ not depending on $k\in I_{\test}.$
\end{proof}

\begin{proposition}
\label{proposition:assumption_concentration_test_holds}
If $(X_i,Y_i)_{i=1}^n$ is stationary $\beta$-mixing, then condition (\ref{eq:hypothesis-concentration_of_empirical_cdf}) holds
with
\begin{align*}
        \varepsilon_{\test}&= \inf_{(a,m)\in F_{\test}}\left\{ \widetilde{\sigma}(a)\sqrt{\frac{4}{n_{\test}} \log\left(\frac{4}{\delta_{\test}-4(m-1)\beta(a)-\beta(n_{\cal})}\right)}\right.\\
    &\left.+ \frac{1}{3m}\log\left(\frac{4}{\delta_{\test}-4(m-1)\beta(a)-\beta(n_{\cal})}\right)+\frac{s}{n_{\test}}\right\}
\end{align*}
for
\[F_{\test} = \{(a,m,s)\in \N_{>0}^2\times\N\ :\ s+2ma = n_{\test},\ \delta>4(m-1)\beta(a)+\beta(n_{\cal})\},\]
and $\widetilde{\sigma}(a)=\sqrt{\frac{1}{4} +\frac{2}{a}\sum_{k=1}^{a-1}(a-k)\beta(k)}.$
\end{proposition}

\begin{proof}
The proof is similar to the proof of Proposition \ref{proposition:assumption_concentration_cal_holds}.
Let the event $E(\varepsilon)$ be
\[E(\varepsilon) =\left\{\left|P_{q,\train}-\frac{1}{n_{\test}}\sum_{i\in I_{\test}}\Ind{\widehat{s}_{\train}(X_i,Y_i)\leq q_{\train }}\right|  > \varepsilon\right\}.\]
Define,
$\P_* = \P_{1}^{n_{\train}} \otimes \P_{n_{\train}+n_{\cal}}^{n_{\train}+n_{\cal}+n_{\test}},$
so under $\P_*$ we have that $(X_i,Y_i)_{i\in I_{\train}}$ and $(X_i,Y_i)_{i\in I_{\test}}$ are independent. By Proposition \ref{prop:blocking_technique} we have
\begin{align*}
    \P[E(\varepsilon)]&\leq \P_*[E(\varepsilon)] +  \beta(n_{\cal})= \E_{*}[\P_{*}[E(\varepsilon)\mid (X_i,Y_i)_{i\in I_{\train}}]] + \beta(n_{\cal}).
\end{align*}
Now we can apply Lemma \ref{lemma:bound_beta_mixing-bernstein}  and conclude that, just as we did in Proposition \ref{proposition:assumption_concentration_cal_holds}, that for any $m, a\in \N_{+}$, $s\in\N$ with  $n_{\test}=2ma-s$ and $\delta_{\test}>4(m-1)\beta(a))+\beta(n_{\cal})$, it is true that
$\P[E(\varepsilon)]\leq \delta$, where
\begin{align*}
    \varepsilon&= \widetilde{\sigma}(a)\sqrt{\frac{4}{n_{\test}} \log\left(\frac{4}{\delta_{\test}-4(m-1)\beta(a)-\beta(n_{\cal})}\right)}\\
    &+ \frac{1}{3m}\log\left(\frac{4}{\delta_{\test}-4(m-1)\beta(a)-\beta(n_{\cal})}\right)+\frac{s}{n_{\test}},
\end{align*}
and
$\widetilde{\sigma}(a)=\sqrt{\frac{1}{4} +\frac{2}{a}\sum_{k=1}^{a-1}(a-k)\beta(k)}.$
Finally, since this is true for any choice of $a,m\in\N_{>0}$ and $s\in\N$, with $s+2ma=n_{\test}$ and $\delta_{\test}>4(m-1)\beta(a)+\beta(n_{\cal})$, we can choose $a,m,s$ optimally to minimize $\varepsilon$.
\end{proof}

\begin{proof}[Theorem \ref{theorem:general_framework_1_beta_mixing}]
The result follows from applying Propositions \ref{proposition:assumption_concentration_cal_holds} and \ref{proposition:assumption_test_to_iid_holds} in Theorem \ref{theorem:marginal_coverage_test_data}.
\end{proof}

\begin{proof}[Theorem \ref{theorem:general_framework_2_beta_mixing}]
The result follows from applying Propositions \ref{proposition:assumption_concentration_cal_holds} and \ref{proposition:assumption_concentration_test_holds} in Theorem \ref{theorem:empirical_coverage_test_data}.
\end{proof}

\subsection{Conditional Guarantees}

To obtain conditional guarantees for stationary $\beta$-mixing processes, we need to specify a family $\sA$ of measurable sets in $\sX$ satisfying certain conditions. In particular, we will assume that, for a fixed value  $\gamma>0$, the family $\sA$ has finite $\VC$ dimension $\VC(\sA)=d$ and $\P[X_*\in A]>\gamma $ for  all $A\in \sA$.

Then, given $\delta_{\cal}>0$ and $\alpha\in(0,1)$, define the calibration error correction factor for a stationary $\beta$-mixing process conditioned to the family $\sA$ as
\begin{align}\label{eq:eps_cal_CCP}
    \varepsilon_{\cal} =\inf_{(a,m,r)\in G_{\cal}}&\left\{\frac{1}{\gamma}\left(\frac{\kappa(m,r)}{n_{\cal}} + \sqrt{\frac{2}{m} \log\left(\frac{16}{\delta_{\cal}-16(m-1)\beta(a)-\beta(r)}\right)}\right)\right\},
\end{align}
where $\kappa(m,r)=4 n_{\cal}\sqrt{\log(2(m+1)^{d})/m} + 2(r-2)$ and
\[G_{\cal}=\{(a,m,r)\in \N_{>0}^3\ :\ 2ma =n_{\cal}-r+1, \delta_{\cal}>16(m-1)\beta(a)+\beta(r)\}.\]
Note the factor $1/\gamma$ in $\varepsilon_{\cal}$: for $\eta$ to be small, we need  $\varepsilon_{\cal}$ to be small and consequently $m$ has to be large. This is quite natural, since if $\gamma$ is too small, the probability $\P[X_*\in A]$ can be close to zero, and thus a larger sample is necessary to estimate the empirical quantile well.
Therefore, Theorem \ref{theorem:marginal_coverage_test_data-conditional} yields the following.
\begin{theorem}[Conditional coverage: stationary $\beta$-mixing processes]\label{theorem:marginal_conditional_beta_mixing}
Suppose that $(X_i,Y_i)_{i=1}^n$ is stationary $\beta$-mixing. Then  given $\alpha\in(0,1)$, $\gamma>0$ and $\delta_{\cal}>0$,  for each $A\in \sA$ and any $i\in I_{\test}$
  \[\P[Y_i\in C_{1-\alpha}(X_i;A)\mid X_i\in A]\geq 1-\alpha - \eta,\]
with $\eta = \varepsilon_{\cal}+\varepsilon_{\test}$, where $\varepsilon_{\cal}$ is as in (\ref{eq:eps_cal_CCP}) and   $\varepsilon_{\test} = \beta(i-n_{\train}).$

Additionally, if $\widehat{s}_{\train}(X_*,Y_*)$ almost surely has a continuous distribution conditionally on the training data, then:
\[|\P[Y_i\in C_{1-\alpha}(X_i;A)\mid X_i\in A]- (1-\alpha)|\leq  \varepsilon_{\cal} + \delta_{\cal} + \varepsilon_{\test}.\]
\end{theorem}

Now, denote the test error correction factor for a stationary $\beta$-mixing process conditioned to the family $\sA$ as
\begin{align}\label{eq:eps_test_CCP}
    \varepsilon_{\test}=\inf_{(a,m,s)\in G_{\test}} &\left\{\frac{1}{\gamma}\left(\frac{\tilde{\kappa}(m,r)}{n_{\test}}+ \sqrt{\frac{2}{m} \log\left(\frac{8}{\delta_{\test}-8(m-1)\beta(a)-\beta(n_{\cal})}\right)}\right)\right\},
\end{align}
where $\tilde{\kappa}(m,r)=4 n_{\test} \sqrt{\log(2(m+1)^{d})/m}+2s$ and
\[G_{\test}=\{(a,m,s)\in \N_{>0}^2\times \N\ :\ 2ma = n_{\test} - s, \delta_{\test}>8(m-1)\beta(a)+\beta(n_{\cal})\}.\]

The following result then follows from Theorem \ref{theorem:empirical_conditional_coverage_test_data}.
\begin{theorem}[Empirical conditional coverage: stationary $\beta$-mixing processes]\label{theorem:conditional_empirica_beta_mixing}
Suppose that $(X_i,Y_i)_{i=1}^n$ is stationary $\beta$-mixing, then given $\alpha\in(0,1)$, $\gamma>0$, $\delta_{\cal}>0$ and $\delta_{\test}>0$, for each $A\in \sA$:
\begin{align*}
    &\P\left[\inf_{A\in\sA}\frac{1}{n_{\test}(A)} \sum_{i \in I_{\test}(A)} \Ind{Y_i\in C_{1-\alpha}(X_i;A)} \geq 1 - \alpha - \eta\right] \geq 1 - \delta_{\cal}-\delta_{\test},
\end{align*}
where $\eta = \varepsilon_{\cal}+\varepsilon_{\test},$ for $\varepsilon_{\cal}$ as in (\ref{eq:eps_cal_CCP}) and $\varepsilon_{\test}$ as in (\ref{eq:eps_test_CCP}).

Additionally, if $\widehat{s}_{\train}(X_*,Y_*)$ almost surely has a continuous distribution conditionally on the training data, then:
\begin{align*}
    &\P\left[\sup_{A\in\sA}\left|\frac{1}{n_{\test}(A)} \sum_{i \in I_{\test}(A)} \Ind{Y_i\in C_{1-\alpha}(X_i;A)} - (1 - \alpha)\right| \leq  \eta\right] \geq 1 - 2\delta_{\cal}-2\delta_{\test}.
\end{align*}
\end{theorem}

The proofs in this section are very similar to the proofs in Section \ref{appendix:beta_mixing-standard_coverage_guarantees}, however, since we are dealing with a family of Borel measurable sets $\sA$, we will need concentration results that allow us to uniformly control  certain quantities over the family $\sA.$ First, we state such classical results for iid sequences.

\begin{theorem}\label{appendix:theorem:vc-inequality} Let $Z_1,\dots,Z_n$ be iid random variables taking values on $\sX$ and $\sF$ be a class of functions from $\sX$ to $\{0,1\}$. Then
\[\E\left[\sup_{f\in \sF}\left|\E f(Z) - \frac{1}{n}\sum_{i=1}^n f(Z_i)\right|\right]\leq 2\sqrt{\frac{\log(2\sS_{\sF}(n))}{n}} .\]
\end{theorem}

Using the Blocking Technique, we can prove that up to a error correction factor, we can transform our stationary $\beta$-mixing problem into a iid one. For example,

\begin{corollary}\label{corollary:uniform_bound_beta_mixing}
    Let $Z_1, \ldots, Z_n$ be a sample drawn from a stationary $\beta$-mixing distribution and $\sF$ be a class of functions from $\sX$ to $\{0,1\}$. Then, for any $a, m,s \in \N_{+}$,  with $m>1$, $n=2ma+s$ and $\delta>4(m-1)\beta(a)$, it holds that
    \begin{equation*}
\P\left[        \sup_{f\in\sF}\left|\E\left[f(Z_1)\right] - \frac{1}{n}\sum_{i=1}^{n} f(Z_i)\right| \leq  \varepsilon_0(a,m,\delta)  \right]\geq 1-\delta,
    \end{equation*}
    where
    \begin{equation}\label{eq:conditional_eps}
    \varepsilon_0(a,m,s,\delta)= 2\sqrt{\frac{\log(2\sS_{\sF}(m))}{m}} + \sqrt{\frac{1}{2m} \log\left(\frac{4}{\delta-4(m-1)\beta(a)}\right)} + \frac{s}{n}.
    \end{equation}
\end{corollary}

\par\noindent{\bf Proof\ }
    By an application of Lemma \ref{lemma:uniform_bound_beta_mixing} and McDiarmids's inequality over the $m$ independent blocks, it follows that
    \begin{align*}
        \P\left[\sup_{f\in\sF}\left| \E[f(Z_1)] - \frac{1}{n} \sum_{i=1}^{n} f(Z_i) \right| > \varepsilon \right]
        &\leq 4 (e^{-2m{\varepsilon'}^2} + (m-1)\beta(a)).
    \end{align*}
where
\begin{align}\label{appendix:eq:t'_expression}
    \varepsilon' &= \varepsilon - \E_*\left[\sup_{f\in\sF}\left|\E[f(Z_1)] - \frac{1}{m} \sum_{j: B_j \in B^{*}_{\mathrm{odd}}}\left(\frac{1}{a}\sum_{i: Z_i \in B_j} f(Z_i)\right) \right|\right]- \frac{s}{n}.
\end{align}
Denote by $Z_j^{(i)}$ the $i$th random variable of the $j$th block $B_j\in B^*_{\textrm{odd}}$, therefore the expectation in (\ref{appendix:eq:t'_expression}) can be written as
\begin{align*}
\E_*\left[\sup_{f\in\sF}\left|\E[f(Z_1)] - \frac{1}{a} \sum_{j=1}^a\left(\frac{1}{m}\sum_{i=1}^m f(Z_j^{(i)})\right) \right|\right] &\leq\frac{1}{a} \sum_{j=1}^a\E_*\left[\sup_{f\in\sF}\left|\E[f(Z_1)] - \frac{1}{m}\sum_{i=1}^m f(Z_j^{(i)})\right|\right],
\end{align*}
where the inequality comes from the triangular inequality and the monotonicity of the supremum.

Note that  in $\frac{1}{m}\sum_{i=1}^m f(Z_j^{(i)})$ we are considering only one element of each independent block $B_j\in B^*_{\textrm{odd}}$, therefore this is a sum over iid random variables.  Hence, by Theorem \ref{appendix:theorem:vc-inequality}
\begin{equation*}
\E_*\left[\sup_{f\in\sF}\left|\E[f(Z_1)] - \frac{1}{m} \sum_{j: B_j \in B^{*}_{\mathrm{odd}}}\left(\frac{1}{a}\sum_{i: Z_i \in B_j} f(Z_i)\right) \right|\right]
\leq 2\sqrt{\frac{\log(2\sS_{\sF}(m))}{m}}.
\end{equation*}
That is,
    $\varepsilon' > \varepsilon - 2\sqrt{\frac{\log(2\sS_{\sF}(m))}{m}}- \frac{s}{n}.$ So taking  $\delta > 4 (e^{-2m{\varepsilon'}^2} + (m-1)\beta(a))$ and $\varepsilon=\varepsilon_0(a,m,\delta)$ yields
\[
    \P\left[\sup_{f\in\sF}\left| \E[f(Z_1)] - \frac{1}{n} \sum_{i=1}^{n} f(Z_i) \right| >  \varepsilon_0(a,m,\delta)\right]  \leq \delta.\tag*{\BlackBox}
\]
\begin{corollary}\label{corollary:uniform_bound_beta_mixing_for_x_in_C}
    Let $(X_*,Y_*), \ldots, (X_n,Y_n)$ be a sample drawn from a stationary $\beta$-mixing distribution, $s:\sX\times \sY \to \R$ be any deterministic function and $\sA$ be a family of Borel measurable sets in $\sX$ with finite $\VC$ dimension $\VC(\sA)=d$.

    Then, for any $m, a \in \N_{+}$ with $m>1$, $n=2ma$ and $\delta>4(m-1)\beta(a)$, it holds that
    \begin{equation*}
    \P\left[    \sup_{{\substack{A\in \sA}}} \left|\P\left[X_*\in A\right] - \frac{1}{n}\sum_{i=1}^{n} \Ind{X_i\in A}\right|\leq \varepsilon_0(a,m,\delta)\right] \geq 1-\delta,
    \end{equation*}
where $\varepsilon_0(a,m,\delta)$ is as defined in (\ref{eq:conditional_eps}).
\end{corollary}

\begin{proof}
Taking
$
    \sF = \left\lbrace x\mapsto \Ind{X_*\in A}\ :\ A\in \sA \right\rbrace
$
in Corollary \ref{corollary:uniform_bound_beta_mixing} and using Sauer-Shelah lemma \citep{sauer1972density, mohri2018adaptive}
yields the result.\end{proof}

\begin{lemma}\label{appendix:lemma:bound_for_infimum_less_than}
     Let $X_1, \ldots, X_n$ be a sample drawn from a stationary $\beta$-mixing distribution, $\gamma\in(0,1)$ and $\sA$ be a family of Borel measurable sets in $\sX$ with finite $\VC$ dimension $\VC(\sA)=d$ such that $\P[X_*\in A]>\gamma$ for  all $A\in \sA$.
          For   $m, a \in \N_{+}$ with $m>1$, $n=2ma$ and $\delta > 4(m-1)\beta(a)$ suppose that
     $\frac{2}{\gamma}\varepsilon_0(a,m,\delta)<1,$ with $\varepsilon_0(a,m,\delta)$ as in (\ref{eq:conditional_eps}).
    Then,
\begin{equation*}
\P\left[\inf_{A\in \sA}\frac{1}{n}\sum_{i=1}^n \Ind{X_i\in A}> \frac{\gamma}{2}\right] \geq 1-\delta.
\end{equation*}
\end{lemma}
\par\noindent{\bf Proof\ }By Corollary \ref{corollary:uniform_bound_beta_mixing_for_x_in_C}, for any $m, a \in \N_{+}$ with $m>1$, $n=2ma$ and $\delta>4(m-1)\beta(a)$, using the fact that $\varepsilon_0(a,m,\delta)<\gamma/2$,
\begin{align*}
    \P\left[\inf_{A\in \sA}\frac{1}{n}\sum_{i=1}^n \Ind{X_i\in A}\leq \frac{\gamma}{2}\right]
     &=     \P\left[\sup_{A\in \sA}\gamma-\frac{1}{n}\sum_{i=1}^n \Ind{X_i\in A}\geq \frac{\gamma}{2}\right]\\
          &\leq  \P\left[\sup_{A\in \sA}\left|\P[X_*\in A]-\frac{1}{n}\sum_{i=1}^n \Ind{X_i\in A}\right|\geq\frac{\gamma}{2}\right]\\
        &\leq  \P\left[\sup_{A\in \sA}\left|\P[X_*\in A]-\frac{1}{n}\sum_{i=1}^n \Ind{X_i\in A}\right|\geq\varepsilon_0(a,m,\delta)\right]\\
          &\leq \delta.\tag*{\BlackBox}
\end{align*}

\begin{lemma}\label{appendix:lemma:uniform_bound_for_conditional_quantile}
     Let $(X_*,Y_*), \ldots, (X_n,Y_n)$ be a sample drawn from a stationary $\beta$-mixing distribution, $s:\sX\times \sY \to \R$ be a deterministic function, $\gamma\in(0,1)$  and $\sA$ be a family of Borel measurable sets in $\sX$ with finite $\VC$ dimension $\VC(\sA)=d$ such that $\P[X_*\in A]>\gamma$ for  all $A\in \sA$.
    For   $m, a \in \N_{+}$ with $m>1$, $n=2ma$ and $\delta>8(m-1)\beta(a)$, if
    $\varepsilon \coloneqq \frac{2}{\gamma}\varepsilon_0(a,m,\delta/2)<1$, for $ \varepsilon_0(a,m,\delta/2)$ as in (\ref{eq:conditional_eps}),
then    with probability at least  $1-\delta$
\[\sup_{{\substack{A\in \sA\\ t\in \R}}}\left|\frac{\P[s(X_*,Y_*)\leq t,X_*\in A]}{\P[X_*\in A]} - \frac{\sum_{i=1}^n \Ind{s(X_i,Y_i)\leq t}\Ind{X_i\in A}}{\sum_{i=1}^n \Ind{X_i\in A}} \right| \leq  \varepsilon.\]

\end{lemma}

\begin{proof}
Define $\varepsilon$ as in the lemma statement. We want to show that:
\[C = \left\lbrace \sup_{{\substack{A\in \sA\\ t\in \R}}}\left|\frac{\P[s(X_*,Y_*)\leq t,X_*\in A]}{\P[X_*\in A]} - \frac{\sum_{i=1}^n \Ind{s(X_i,Y_i)\leq t}\Ind{X_i\in A}}{\sum_{i=1}^n \Ind{X_i\in A}} \right| > \varepsilon\right\rbrace\]
has probability at most $\delta$. To this end, we define the following auxiliary event, which controls the random denominator term in $C$:
\[B = \left\lbrace \inf_{A\in \sA}\frac{1}{n}\sum_{i=1}^n \Ind{X_i\in A}>\frac{\gamma}{2}\right\rbrace.\]

By Lemma \ref{appendix:lemma:bound_for_infimum_less_than}, $\P[B^c]<\delta/2$, so it suffices to show that $\P[E]\leq \frac{\delta}{2}$ where $E \coloneqq C\cap B.$

Note that, if $E$ holds, then the quotient $\frac{\sum_{i=1}^n \Ind{s(X_i,Y_i)\leq t}\Ind{X_i\in A}}{\sum_{i=1}^n \Ind{X_i\in A}}$ is well defined and
\begin{align*}
    \varepsilon&< \sup_{{\substack{A\in \sA\\ t\in \R}}}\left|\frac{\P[s(X_*,Y_*)\leq t,X_*\in A]}{\P[X_*\in A]} - \frac{\sum_{i=1}^n \Ind{s(X_i,Y_i)\leq t}\Ind{X_i\in A}}{\sum_{i=1}^n \Ind{X_i\in A}} \right|\\
    &\leq  \sup_{{\substack{A\in \sA\\ t\in \R}}}\left|\frac{\P[s(X_*,Y_*)\leq t,X_*\in A] - \frac{1}{n}\sum_{i=1}^n \Ind{s(X_i,Y_i)\leq t}\Ind{X_i\in A}}{\P[X_*\in A]}\right|\\
    &+  \sup_{{\substack{A\in \sA\\ t\in \R}}}\left|\frac{\sum_{i=1}^n \Ind{s(X_i,Y_i)\leq t}\Ind{X_i\in A}\left(\P[X_*\in A] - \frac{1}{n}\sum_{i=1}^n \Ind{X_i\in A}\right)}{\P[X_*\in A]\sum_{i=1}^n \Ind{X_i\in A}}\right|
    \end{align*}
    \begin{align*}
    &\leq \sup_{{\substack{A\in \sA\\ t\in \R}}}\left|\frac{\P[s(X_*,Y_*)\leq t,X_*\in A] - \frac{1}{n}\sum_{i=1}^n \Ind{s(X_i,Y_i)\leq t}\Ind{X_i\in A}}{\gamma}\right|\\
    &+  \sup_{{\substack{A\in \sA}}}\left|\frac{\P[X_*\in A] - \frac{1}{n}\sum_{i=1}^n \Ind{X_i\in A}}{\gamma}\right|,
\end{align*}
Moreover, for any $A\in \sA$:
\begin{align*}
    &\left|\frac{\P[X_*\in A] - \frac{1}{n}\sum_{i=1}^n \Ind{X_i\in A}}{\gamma}\right| \\
    &= \lim_{t\to +\infty}\left|\frac{\P[s(X_*,Y_*)\leq t,X_*\in A] - \frac{1}{n}\sum_{i=1}^n \Ind{s(X_i,Y_i)\leq t}\Ind{X_i\in A}}{\gamma}\right|
\end{align*}
We deduce that:
\begin{equation}\label{eq:Eholds}\mbox{$E$ holds }\Rightarrow \varepsilon<2\,\sup_{{\substack{A\in \sA}}}\left|\frac{\P[X_*\in  A] - \frac{1}{n}\sum_{i=1}^n \Ind{X_i\in A}}{\gamma}\right|.
\end{equation}
By Corollary  \ref{corollary:uniform_bound_beta_mixing_for_x_in_C}, we know that, with our choice of $\varepsilon$:
\[\P\left\{\sup_{{\substack{A\in \sA}}}\left|\frac{\P[X_*\in A] - \frac{1}{n}\sum_{i=1}^n \Ind{X_i\in A}}{\gamma}\right|>\frac{\varepsilon}{2}\right\}\leq \frac{\delta}{2}.\]
By (\ref{eq:Eholds}), we also have $\P[E]\leq \delta/2$. This finishes the proof.
\end{proof}

\begin{proposition}\label{proposition:assumption_concentration_cal_holds_conditional}
Let
\[
    \varepsilon =\inf_{(a,m,r)\in G_{\cal}}\left\{\frac{2}{\gamma}\left(\varepsilon_0\left(a,m,\frac{\delta-\beta(r)}{4}\right)+\frac{2(r-1)}{n_{\cal}}\right)\right\}
\]
where $ \varepsilon_0(a,m,\delta/2)$ as in (\ref{eq:conditional_eps}) and
\[G_{\cal}=\{(a,m,r)\in \N_{>0}^3\ :\ 2ma =n_{\cal}-r+1, \delta>16(m-1)\beta(a)+\beta(r)\}.\]
If $\varepsilon<1$, then condition (\ref{eq:hypothesis-concentration_over_calibration-conditional}) holds with $\varepsilon_{\cal} = \varepsilon.$

\end{proposition}
\begin{proof}
The proof is similar to the proof of Proposition \ref{proposition:assumption_concentration_cal_holds}.
For $\varepsilon>0$ and $r\in \{1,\dots,n_{\cal}\}$, let $I_{\cal,r}=\{n_{\train}+r,\dots,n_{\train}+n_{\cal}\}$ and
$I_{\cal,r}(A) = \{i\in I_{\cal,r}\ :\ X_i\in A\}.$
Define the events

\begin{align*}
    E(r,\varepsilon') =\left\{\inf_{A\in\sA}\left|\frac{\sum_{i\in I_{\cal,r}(A)}\Ind{\widehat{s}_{\train}(X_i,Y_i)\leq q_{\train }}}{\# I_{\cal,r}(A)} - P_{q,\train}(A)\right|> \varepsilon'\right\},\\
    C =\left\{\inf_{A\in \sA}\frac{1}{n_{\cal}}\sum_{i\in I_{\cal}} \Ind{X_i\in A}> \frac{\gamma}{2}\right\}\quad \textrm{and} \quad B(r,\varepsilon') = E(r,\varepsilon')\cap C.
\end{align*}

We want to show that there exists $\varepsilon'>0$ such that if $\varepsilon'<1$ then  $\P[E(1,\varepsilon')]\leq \delta.$
Note that if $B(1,\varepsilon')$ holds, then for all $A\in \sA$,
\[
    \left|\frac{\sum_{i\in I_{\cal,r}(A)}\Ind{\widehat{s}_{\train}(X_i,Y_i)\leq q_{\train }}}{\#I_{\cal,r}(A)}- P_{q,\train}(A)\right|> \varepsilon'- \frac{2(r-1)}{\gamma n_{\cal}}.
\]
That is, $B(1,\varepsilon')\subset B\left(r,\varepsilon'-\frac{2(r-1)}{\gamma n_{\cal}}\right)$. Now,  define $\P_* = \P_{1}^{n_{\train}} \otimes \P_{n_{\train}+r}^{n_{\train}+n_{\cal}},$
so under $\P_*$ we have that $(X_i,Y_i)_{i\in I_{\train}}$ and $(X_i,Y_i)_{i\in I_{\cal},r}$ are independent. By Proposition \ref{prop:blocking_technique} we have
\begin{equation*}
    \P[B(1,\varepsilon')]\leq \P\left[B\left(r,\varepsilon'-\frac{2(r-1)}{\gamma n_{\cal}}\right)\right]\leq \P_*\left[B\left(r,\varepsilon'-\frac{2(r-1)}{\gamma n_{\cal}}\right)\right] + \beta(r).
\end{equation*}
But this implies that
$\P[E(1,\varepsilon')]\leq \P_*\left[B\left(r,\varepsilon'-\frac{2(r-1)}{\gamma n_{\cal}}\right)\right] + \beta(r) + 1- \P[C].$
For any $m, a\in \N_{+}$ with $n_{\cal}-r+1=2ma$ and $\delta_{\cal}>8(m-1)\beta(a)$, if we take
\begin{equation*}
    \varepsilon' =\frac{1}{\gamma}\left(4\sqrt{\frac{\log(2(m+1)^{d})}{m}} + 2\sqrt{\frac{1}{2m} \log\left(\frac{8}{\delta-8(m-1)\beta(a)}\right)}+\frac{2(r-1)}{n_{\cal}}\right),
\end{equation*}
and assume that $\varepsilon'<1$ then
\[\frac{1}{\gamma}\left(4\sqrt{\frac{\log(2(m+1)^{d})}{m}} + 2\sqrt{\frac{1}{2m} \log\left(\frac{8}{\delta-8(m-1)\beta(a)}\right)}\right)< 1\]
so Lemma \ref{appendix:lemma:uniform_bound_for_conditional_quantile} tells us that
$\E_*\left[\P_*\left[B\left(r,\varepsilon'-\frac{2(r-1)}{\gamma n_{\cal}}\right)\mid (X_i,Y_i)_{i\in I_{\train}}\right]\right]\leq \delta$
and Lemma \ref{appendix:lemma:bound_for_infimum_less_than} tells us  $1- \P[C]\leq \delta$.
That is,
$\P[E(1,\varepsilon')]\leq 2\delta + \beta(r),$
which is equivalent to
$\P[E(1,\varepsilon)]\leq \delta,$
if $\varepsilon$ is as in the proposition statement.
\end{proof}

\begin{proposition}\label{proposition:assumption_test_to_iid_holds_conditional}
Condition (\ref{eq:hypothesis-decoupling_over_test_data-conditional}) holds  with
$\varepsilon_{\train} = \beta(k-n_{\train}).$
\end{proposition}

\begin{proof}
Given $k\in I_{\test}$, note that we can decompose $\P_* = \P_{1}^{n_{\train}} \otimes \P_{k}^{k},$
so under $\P_*$ we have that $(X_k,Y_k)$ is independent of $(X_i,Y_i)_{i\in I_{\train}}.$
Then, defining
$\beta_k = \beta(k-n_{\train})$ we have for all $A\in \sA,$
    \begin{align*}
        \beta_k&\geq \left|\P[\widehat{s}_{\train}(X_k, Y_k) \leq {q}_{\train}(A),X_k\in A]- \P_*[\widehat{s}_{\train}(X_k, Y_k) \leq {q}_{ \train}(A),X_k\in A]\right|
    \end{align*}
    where the $\beta_k$ penalty follows from Proposition \ref{prop:blocking_technique}.
    But then, by a conditioning argument,
    \begin{align*}
        \beta_k &\geq   \left|  \P[\widehat{s}_{\train}(X_k, Y_k) \leq {q}_{\train}(A),X_k\in A]-\P[\widehat{s}_{\train}(X_*, Y_*) \leq {q}_{ \train}(A),X_*\in A]\right|.
    \end{align*}
Since $\frac{\beta_k}{\P[X_k\in A]}\geq \beta_k$, dividing by $\P[X_k\in A] = \P[X_*\in A]$
yields the result.
\end{proof}

\begin{proposition}\label{proposition:assumption_concentration_test_holds_conditional}
Define
\[
    \varepsilon =\inf_{(a,m,s)\in G_{\test}}\left\{\frac{2}{\gamma}\left(\varepsilon_0\left(a,m,\frac{\delta-\beta(n_{\cal})}{2}\right)\right)+ \frac{s}{n_{\test}}\right\}
\]
where
$G_{\test}=\{(a,m)\in \N_{>0}^2\ :\ s+2ma =n_{\test}, \delta>8(m-1)\beta(a)+\beta(n_{\cal})\}.$
If $\varepsilon<1$, then condition (\ref{eq:hypothesis-concentration_of_empirical_cdf-conditional}) holds with $\varepsilon_{test}=\varepsilon.$
\end{proposition}
\begin{proof}
The proof is similar to the proof of Proposition \ref{proposition:assumption_concentration_cal_holds_conditional}.
Let the event $E(\varepsilon)$ be
\[E(\varepsilon)=\left\{\inf_{A\in\sA}\left|   P_{q,\train}(A)-\frac{\sum_{i\in I_{\test}(A)}\Ind{\widehat{s}_{\train}(X_i,Y_i)\leq q_{\train }]}}{n_{\test}(A)}\right|> \varepsilon\right\},\]
Define, $\P_* = \P_{1}^{n_{\train}} \otimes \P_{n_{\train}+n_{\cal}}^{n_{\train}+n_{\cal}+n_{\test}},$
so under $\P_*$ we have that $(X_i,Y_i)_{i\in I_{\train}}$ and $(X_i,Y_i)_{i\in I_{\test}}$ are independent. By Proposition \ref{prop:blocking_technique} we have
\begin{align*}
    \P[E(\varepsilon)]&\leq \P_*[E(\varepsilon)] +  \beta(n_{\cal})
= \E_{*}[\P_{*}[E(\varepsilon)\mid (X_i,Y_i)_{i\in I_{\train}}]] + \beta(n_{\cal}).
\end{align*}
Now we can apply Lemma \ref{appendix:lemma:uniform_bound_for_conditional_quantile}  and conclude that if $\varepsilon<1$, for any $m, a\in \N_{+}$ with  $n_{\test}=2ma$ and $\delta_{\test}>8(m-1)\beta(a))+\beta(n_{\cal})$, it is true that
$\P[E]\leq \delta$, where
\[\varepsilon=\frac{1}{\gamma}\left(4\sqrt{\frac{\log(2(m+1)^{d})}{m}} + 2\sqrt{\frac{1}{2m} \log\left(\frac{8}{\delta-8(m-1)\beta(a)-\beta(n_{\cal})}\right)}+\frac{s}{n_{\test}}\right).\]
Finally, since this is true for any choice of $a,m,s\in\N_{>0}$ with $s+2ma=n_{\test}$ and $\delta_{\test}>8(m-1)\beta(a)+\beta(n_{\cal})$, we can choose $a,m,s$ optimally to minimize $\varepsilon$.
\end{proof}

\begin{proof}[Theorem \ref{theorem:marginal_conditional_beta_mixing}]
Follows from applying Propositions \ref{proposition:assumption_concentration_cal_holds_conditional} and \ref{proposition:assumption_test_to_iid_holds_conditional} in Theorem \ref{theorem:marginal_coverage_test_data-conditional}.
\end{proof}

\begin{proof}[Theorem \ref{theorem:conditional_empirica_beta_mixing}]
Follows from applying Propositions \ref{proposition:assumption_concentration_cal_holds_conditional} and \ref{proposition:assumption_concentration_test_holds_conditional} in Theorem \ref{theorem:empirical_conditional_coverage_test_data}.
\end{proof}

\section{Proofs of Section \ref{sec:extensions} and Further Results}
\label{appendix:extensions}

\subsection{Non-Stationary Data}\label{sub:nonstationary}

\begin{proof}[Theorem \ref{theorem:marginal_nonstat}] This proof is similar to that of Theorem \ref{theorem:marginal_coverage_test_data}, but the notation is somewhat more complicated due to nonstationarity.

Consider $N_{\cal}$ as in the statement of the present Theorem. For $\phi\in (0,1)$, define the conditional $\phi$-quantile of $\widehat{s}_{\train}(X_{*,N_{\cal}},Y_{*,N_{\cal}})$ given the training data:
\[q_{\phi,\train}\coloneqq \inf\{t\in\R\,:\,\P[\widehat{s}_{\train}(X_{*,N_{\cal}},Y_{*,N_{\cal}})\leq t\mid \{(X_i,Y_i)\}_{i\in I_{\train}}]\geq \phi\},
\]
or alternatively,
\[q_{\phi,\train}\coloneqq \inf\left\{t\in\R\,:\,\frac{1}{n_{\cal}}\sum_{j\in N_{\cal}}\P[\widehat{s}_{\train}(X_{*,j},Y_{*,j})\leq t\mid \{(X_i,Y_i)\}_{i\in I_{\train}}]\geq \phi\right\}.
\]
For each $j\in I_{\cal}$, set
$p^{(j)}_{\phi,\train}\coloneqq \P[\widehat{s}_{\train}(X_{*,j},Y_{*,j})\leq q_{\phi_\ell,\train}\mid \{(X_i,Y_i)\}_{i\in I_{\train}}].
$
Fix $i\in I_{\test}$. Following the proof of Theorem \ref{theorem:marginal_coverage_test_data}, we consider values of $\phi$ of the form:
\[\phi_\ell \coloneqq 1-\alpha-\varepsilon_{\cal}-\delta^{(i)}-1/\ell\mbox{ for }\ell=1,2,3\dots\]
to obtain that $\P[F]\geq 1-\delta_{\cal}$, where $F\coloneqq\left\{q_{1-\alpha-\varepsilon_{\cal}.\cal}\leq \widehat{q}_{1-\alpha,\cal}\right\}.$
Now,
  \begin{align*}
  \P[\widehat{s}_{\train}(X_i,Y_i)\leq \widehat{q}_{1-\alpha,\cal}]  \geq &  \P[\widehat{s}_{\train}(X_i,Y_i)\leq q_{1-\alpha-\varepsilon_{\cal}}] - \P[F^c] \\ \geq &  \P[\widehat{s}_{\train}(X_{*,i},Y_{*,i})\leq q_{1-\alpha-\varepsilon_{\cal}}] -\delta_{\cal} - \varepsilon_{\test}.
  \end{align*}
By assumption, the law of $(X_{*,i},Y_{*,i})$ is $\delta^{(i)}$-close in total variation to that of $(X_{*,N_{\cal}},Y_{*,N_{\cal}})$. Since the event $\{\widehat{s}_{\train}(X_{*,i},Y_{*,i})\leq q_{1-\alpha-\delta^{(i)}-\varepsilon_{\cal}}\}$ depends on $(X_{*,i},Y_{*,i})$ and on the independent process $(X_j,Y_j)_{j\in[n]}$, and $q_{1-\alpha-\varepsilon_{\cal},\train}$ is the conditional ($1-\alpha-\varepsilon_{\cal}$)-quantile of $\widehat{s}_{\train}(X_{*,N_{\cal}},Y_{*,N_{\cal}})$ given the training data, one may conclude:
\begin{align*}
\P[\widehat{s}_{\train}(X_{*,i},Y_{*,i})\leq q_{1-\alpha-\varepsilon_{\cal}}]\geq &  \P[\widehat{s}_{\train}(X_{*,N_{\cal}},Y_{*,N_{\cal}})\leq q_{1-\alpha-\varepsilon_{\cal}}] -\delta^{(i)}\\ \geq &  1-\alpha-\varepsilon_{\cal}-\delta^{(i)}.
\end{align*}
Plugging this back into the previous display finishes the proof. \end{proof}

\subsection{Risk-Controlling Prediction Sets}\label{sub:RCPS}

Risk-controlling prediction sets (RCPS), introduced by \citet{bates2021rcps}, give a general methodology for CP that applies in a variety of settings, including regression, multiclass classification and image segmentation. Importantly, RCPS does not involve nonconformity scores, but rather, the construction of nested sets. While the original theory of RCPS assumes independent data, we now show it also applies within our framework.

Suppose $\sY'$ is a family of sets, $\Lambda\subset \R\cup\{+\infty\}$ is a closed set, and a map $\sT\colon(\sX\times \sY)^{n_{\train}}\times \sX\times \Lambda\to \sY'$ is given with the following property: for all choices of $(x_i,y_i)_{i=1}^{n_{\cal}}\in (\sX\times \sY)^{n_{\cal}}$, $x\in \sX$ and $\lambda_1,\lambda_2\in \Lambda$: if $\lambda_1 \leq \lambda_{2}$, then $\sT((x_i,y_i)_{i=1}^{n_{\cal}},x,\lambda_1) \subset \sT((x_i,y_i)_{i=1}^{n_{\cal}},x,\lambda_2)$.

For $(x,\lambda)\in \sX$, we use the notation
\[\widehat{\sT}_{\lambda,\train}(x) \coloneqq \sT((X_i,Y_i)_{i\in I_{\train}},x,\lambda)\]
to denote the values of $\sT$ when the first $n_{\train}$ pairs in the input correspond to the training data. We call $\widehat{\sT}_{\lambda,\train}(\cdot)$ a trained tolerance region. Finally, $L\colon\sY\times \sY'\to \R$ is a loss function that is decreasing in the $\sY'$ component. The goal of RCPS is to compute a value $\widehat{\lambda}$ from the calibration data that achieves (conditional) risk smaller than a prespecified level $\alpha>0$.

To define the conditional risk, assume that the map from $\lambda\in\Lambda$ to $\E[L(Y_*,\widehat{\sT}_{\lambda}(X_*))\mid (X_i,Y_i)_{i\in I_{\train}}]$ almost surely is continuous and achieves arbitrarily small positive values. Given a measurable $\ell\colon(\sX\times \sY)^{n_{\train}}\to \Lambda$, we let $\ell_{\train} \coloneqq \ell((X_i,Y_i)_{i\in I_{\train}})$, define the conditional expected risk as
\begin{equation*}
    R(\ell) \coloneqq \E[L\big(Y_*,\widehat{\sT}_{\ell_{\train}}(X_*)\big)\mid (X_i,Y_i)_{i\in I_{\train}}].
\end{equation*}
Also, define the empirical risk over the calibration data as
\begin{equation*}
    \widehat{R}_{\cal}(\lambda) \coloneqq \frac{1}{n_{\cal}}\sum_{i\in I_{\cal}}L(Y_i,\sT_{\lambda}(X_i)).
\end{equation*}

Now, a threshold $\widehat{\lambda}$ must be chosen from calibration data to control the risk. In \citet{bates2021rcps}, this requires finding a function $\lambda\mapsto \widehat{R}_{\text{UCB}}(\lambda)$ that gives a pointwise high-probability upper bound on $R(\lambda)$. In our case, we can allow for a $\widehat{R}(\lambda)$ that bounds the risk up to a small error; for us, the empirical risk will play this role. Thus, consider the empirical threshold
\begin{equation*}
    \widehat{\lambda}_{\alpha,\cal} \coloneqq \inf\left\{\lambda\in \Lambda\,:\, \forall \lambda'\in \Lambda, \lambda'>\lambda\implies \widehat{R}_{\cal}(\lambda)<\alpha\right\}.
\end{equation*}

Finally, we give conditions that guarantee that $\widehat{\lambda}_{\alpha,\cal}$ controls the risk with high probability. First, assume that there exist $\varepsilon_{\cal}>0$, $\delta_{\cal}\in (0,1)$ such that, for any $\ell,\ell_{\train}$,
\begin{equation}\label{eq:rcps-hypothesis1}
    \P\left[|\widehat{R}_{\cal}(\ell_{\train}) - R(\ell)|\leq \varepsilon_{\cal}\right]\geq 1-\delta_{\cal}.
\end{equation}
Also, assume there exists a $\varepsilon_{\test}$ such that for all $i\in I_{\test}$ and all $\ell$,
\begin{equation} \label{eq:rcps-hypothesis2}
    |\E[L(Y_i,\sT_{\ell_{\train}}(X_i))] -\E[R(\ell)]|\leq \varepsilon_{\test}.
\end{equation}

Then, the following result on the performance of RCPS over a single test point holds.

\begin{theorem}[Approximate risk control for $\widehat{\lambda}_{\alpha,\cal}$]\label{thm:RCPS} Assume (\ref{eq:rcps-hypothesis1}) and (\ref{eq:rcps-hypothesis2}). Then,
\begin{equation*}
    \P\left[\E[L(Y_*,\sT_{\widehat{\lambda}_{\alpha,\cal}}(X_*))]\leq \alpha + \varepsilon_{\cal}\right]\geq 1-\delta_{\cal}.
\end{equation*}
Moreover, if $L$ is uniformly bounded, we have the following for all $i\in I_{\test}$:
\begin{equation*}
    \E[L(Y_i,\sT_{\widehat{\lambda}_{\alpha,\cal}}(X_i))]\leq \alpha + \varepsilon_{\test} + \varepsilon_{\cal} + \|L\|_{\infty}\,\delta_{\test}.
\end{equation*}
\end{theorem}
\begin{proof}
The proof is a combination of our arguments for Theorem \ref{theorem:marginal_coverage_test_data} with the reasoning in \citet{bates2021rcps}. Recall that for any function $\ell:(\sX\times\sY)^{n_{\train}}\to \Lambda$, if $\ell_{\train} \coloneqq \ell((X_i,Y_i)_{i\in I_{\train}})$,
\[R(\ell) \coloneqq \E[L(Y_*,\sT_{\ell_{\train}}(X_*))\mid (X_i,Y_i)_{i\in I_{\train}}].\]
Make the specific choice
$\ell_{\train}\coloneqq \inf\{\lambda\in \Lambda\,:\,R(\lambda)\leq \alpha+\varepsilon_{\cal}\},$
so that, by right-continuity of $\lambda\mapsto \E[L(Y_*,\sT_{\lambda}(X_*))\mid (X_i,Y_i)_{i\in I_{\train}}]$,
\begin{equation}\label{eq:continuityRCPS}R(\ell_{\train})\leq \alpha+\varepsilon_{\cal}\end{equation}
while at the same time $R(\ell_{\train}-1/k)>\alpha + \varepsilon_{\cal}\mbox{ for all $k\in\N$.}$ The definition of $\widehat{\lambda}_{\alpha,\cal} = \inf\{\lambda\in \Lambda\,:\,\widehat{R}_{\cal}(\lambda)<\alpha\}$ and the fact that our risk decreases with $\lambda$ imply:
\[\P[\widehat{\lambda}_{\alpha,\cal}\geq \ell_{\train} - 1/k]\geq \P[\widehat{R}_{\cal}(\ell_{\train} - 1/k)>\alpha]\geq 1-\delta_{\cal}\]
where the last step uses assumption (\ref{eq:rcps-hypothesis1}). Letting $k\to +\infty$ gives:
\begin{equation}\label{eq:widehatlambda}\P[\widehat{\lambda}_{\alpha,\cal}\geq \ell_{\train}]\geq 1-\delta_{\cal}\end{equation}
Now, $\widehat{\lambda}_{\alpha,\cal}\geq \ell_{\train}$ and (\ref{eq:continuityRCPS}) together imply:
\begin{equation}\label{eq:contin}\E[L(Y_*,\sT_{\widehat{\lambda}_{\alpha,\cal}}(X_*))]\leq \E[L(Y_*,\sT_{\ell_{\train}}(X_*))] = \E[R(\ell_{\train})]\leq \alpha+\varepsilon_{\cal},\end{equation}
so the first assertion in the Theorem follows from (\ref{eq:widehatlambda}). The second assertion follows from:
\[\E[L(Y_i,\sT_{\widehat{\lambda}_{\alpha,\cal}}(X_i))]\leq \E[L(Y_i,\sT_{\ell_{\train}}(X_i))] + \|L\|_{\infty}\P[\widehat{\lambda}_{\alpha,\cal}<\ell_{\train}]\]
combined with (\ref{eq:rcps-hypothesis2}) and (\ref{eq:contin}).\end{proof}

Thus the expected loss at any test point is controlled by $\alpha$ plus an error term that can be shown to be small, even for non-exchangeable data. Importantly, the result is achieved via assumptions that only bound the behavior of the loss over individual thresholds $\ell_{\train}$ obtained from the training data. In particular, there is no need to require uniform control of the loss over a range of $\ell$, which would require stronger (and looser) concentration bounds. The uniform bound on $L$ can be replaced by a moment assumption, at the cost of a less clean bound.

\subsection{Rank-One-Out Conformal Prediction}\label{sub:ROO}

Rank-one-out (ROO) conformal prediction, introduced by \citet{lei2018distribution}, is different from split CP in that the method calibrates the quantile used for each test data point by looking at the remaining test points.
This requires adapting the above setup as follows: partition the data indices as $[n]=I_{\train}\sqcup I_{\test}$, and for each $i \in I_{\test}$ the calibration set is $I^{(i)}_{\cal}=I_{\test}\setminus\{i\}$. Also, define the empirical quantiles as follows: given $\phi\in [0,1)$ and $i\in I_{\test}$, let $\widehat{q}^{(i)}_{\phi,\cal}$ denote the empirical $\phi$-quantile
\[\widehat{q}^{(i)}_{\phi,\cal} \coloneqq \inf\left\{t\in\R\,:\, \frac{1}{n_{\test}-1}\sum_{j\in I_{\cal}^{(i)}}\Ind{\widehat{s}_{\train}(X_j,Y_j)\leq t}\geq \phi\right\}.\]
For $x\in\sX$, the rank-one-out prediction set for $i\in I_{\test}$ is then defined via:
\[C^{(i)}_{\phi}(x) \coloneqq \{y\in\sY\,:\, \widehat{s}_{\train}(x,y)\leq \widehat{q}^{(i)}_{\phi,\cal}\}.\]

We can then adapt the concentration and decoupling hypotheses. Indeed, we assume there exist $\varepsilon_{\test}\in (0,1)$, $\{\varepsilon_{\test}(i)\}_{i\in I_{\test}}\subset (0,1)$ and $\delta_{\test}\in (0,1)$ such that, for any $i\in I_{\test}$,
\begin{equation} \label{eq:roo-hypothesis1}
    |\P[\widehat{s}_{\train}(X_i,Y_i)\leq q_{\train}] - \P[\widehat{s}_{\train}(X_*,Y_*)\leq q_{\train}]|\leq \varepsilon_{\test}(i),
\end{equation}
and, moreover,
\begin{equation} \label{eq:roo-hypothesis2}
    \P\left[\left|\frac{1}{n_{\test}}\sum_{i\in I_{\test}}\Ind{\widehat{s}_{\train}(X_i,Y_i)\leq q_{\train}} - P_{q,\train}\right|\leq \varepsilon_{\test}\right]\geq 1-\delta_{\test}.
\end{equation}

Then, the analogue of Theorems \ref{theorem:marginal_coverage_test_data} and \ref{theorem:empirical_coverage_test_data} still hold for ROO.

\begin{theorem}[Marginal and empirical coverage over test data for ROO]\label{theorem:marginal_coverage_test_data_ROO}
Given a level $\alpha\in(0,1)$, if (\ref{eq:roo-hypothesis1}) and (\ref{eq:roo-hypothesis2}) hold, then, for all $i\in I_{\test}$:
\[\P[\widehat{s}_{\train}(X_i,Y_i)\leq \widehat{q}_{1-\alpha,\cal}]\geq 1-\alpha -  \varepsilon_{\test}(i) - \varepsilon_{\test} - \delta_{\cal}-\frac{1}{n_{\test}}.\]
Moreover, it holds that
\[\P\left[\frac{1}{n_{\test}}\sum_{i\in I_{\test}}\Ind{\widehat{s}_{\train}(X_i,Y_i)\leq \widehat{q}^{(i)}_{1-\alpha,\cal}}\geq 1-\alpha-\varepsilon_{\test}-\frac{1}{n_{\test}}\right]\geq 1-2\delta_{\test}.\]
\end{theorem}
\begin{proof}
If we consider $I_{\cal} \coloneqq I_{\test}$ in the proof of Theorem \ref{theorem:marginal_coverage_test_data},  the event
$F=\{q_{1-\alpha-\varepsilon_{\cal},\train}\leq \widehat{q}_{1-\alpha,\cal}\}$
satisfies $\P[F]\geq 1-\delta_{\cal}.$
But since
$\widehat{q}^{(i)}_{1-\alpha,\cal}\geq \widehat{q}_{1-\alpha-1/n_{\test},\cal},$
it is also true that the event
$F'=\{q_{1-\alpha-\varepsilon_{\cal}-1/n_{\test},\train}\leq \widehat{q}^{(i)}_{1-\alpha,\cal}\}$
also satisfies $\P[F']\geq 1-\delta_{\cal}$. The rest of the proof follows the same strategy as in Theorem \ref{theorem:marginal_coverage_test_data} using $\widehat{q}^{(i)}_{1-\alpha,\cal}$ instead of $\widehat{q}_{1-\alpha,\cal}$.
\end{proof}

One can adapt the analysis in Section \ref{sec:stationary_beta_mixing_data} to bound the parameters $\delta_{\test}$, $\varepsilon_{\test}$ and $\varepsilon_{\test}(i)$ for $\beta$-mixing data. In particular, one may take $\varepsilon_{\test}(i)=\beta(i-n_{\cal})$, and $\varepsilon_{\test},\delta_{\test}$ equal to the respective parameters $\varepsilon_{\cal},\delta_{\cal}$ in that section, but with $n_{\test}$ replacing $n_{\cal}$ (since the calibration set for each point of rank-one-out is essentially equal to the test set).

On the other hand, we note that marginal coverage might suffer somewhat over the first few test data, since $\varepsilon_{\test}(i)=\beta(i-n_{\cal})$ may be large for small values $i-n_{\cal}$. In contrast to split CP, there is no gap in ROO between training and test data so the first test points may be strongly correlated with the training data.

\section{Details of Section~\ref{sec:empirical_studies}}
\label{appendix:empirical_studies}

For the experiment comparing split CP and EnbPI (cf. Figure~\ref{fig:rolling_coverage-split_cp-enbpi}), split CP's underlying random forest model comprised 100 trees; mean squared error was used as the split criterion; no maximum tree depth was set, so nodes are expanded until all leaves contain less than 2 samples; all features were considered for splitting. EnbPI's random forest was exactly the same and the length of the blocks in the EnbPI's block bootstrap procedure was set to 8 with the number of resamplings set to 30. For the $\text{AR}(1)$ experiment (cf. Figure~\ref{fig:intro}) and Examples 3 and 4, gradient boosting was set to boost 100 trees with a learning rate of 0.1 and pinball loss; trees of any depth were allowed; the minimal number of data in one leaf was 20; the minimal sum hessian in one leaf was 0.001; no minimal gain to perform a split was required; no more than 31 leaves were allowed per tree; no regularization was set. The neural network consisted of three fully connected layers with ReLU activation; the number of output units were 128, 64 and 2, respectively, where the final output of 2 units represents the low and high quantiles being estimated; AdamW with learning rate of $10^{-3}$ and weight decay of $10^{-6}$ was used; training was over 100 epochs with batches of size 64; pinball loss was used. The random forest model (quantile regression forest) comprised 10 trees; mean squared error was used as the split criterion; no maximum tree depth was set, so nodes are expanded until all leaves contain less than 2 samples; all features were considered for splitting. Quantile regressors making use of the pinball loss $L_\tau(y, \hat{y}) = \tau (y - \hat{y}) \Ind{y \geq \hat{y}} + (1-\tau) (\hat{y}-y) \Ind{y < \hat{y}}$ had $\tau$ set to $\alpha/2$ and $1-\alpha/2$, with $\alpha$ the acceptable miscoverage level.

\vskip 0.2in
\bibliography{bibliography}

\end{document}